\documentclass{amsart}
\usepackage{hyperref}
\usepackage{amssymb}
\usepackage{enumitem}
\usepackage{bm}
\usepackage{calrsfs}
\usepackage{comment}
\usepackage{marginnote}
\DeclareMathAlphabet{\pazocal}{OMS}{zplm}{m}{n}

\def\R{\mathbb{R}}
\def\C{\mathbb{C}}
\def\N{\mathbb{N}}
\def\Z{\mathbb Z}

\def\H{\mathbb{H}}
\def\O{\mathbb{O}}
\def\0{{\bar 0}}
\def\1{{\bar 1}}
\def\mat{\mathop{\mathcal{M}}}
\def\der{\mathop{\text{\rm Der}}}
\def\aut{\mathop{\text{\rm Aut}}}
 
\def\esc#1{\langle #1\rangle}

\def\tr{\mathop{\hbox{\rm tr}}}
\def\l{\lambda}
\def\a{\alpha}
\def\b{\beta} 
\def\o{\otimes}
\def\s{\sigma}

\def\iu{\hm{i}}
\def\ju{\hm{j}}
\def\ku{\hm{k}}

\def\sl{\mathfrak{sl}}

\def\SO{\mathop{\hbox{\rm SO}}}

\def\GL{\mathop{\hbox{\rm GL}}}
\def\SL{\mathop{\hbox{\rm SL}}}
\def\SU{\mathop{\hbox{\rm SU}}}
\def\g{\mathfrak{g}}
\def\h{\mathfrak{h}}
\def\m{\mathfrak{m}}
\def\t{\times}
\def\tr{\text{\rm tr}}
\def\ker{\text{\rm ker}}
\def\id{\text{\rm id}}
\def\tp#1{\mathop{\langle #1\rangle}}
\def\grass{\mathop{\text{\bf Gr}}}
\def\M{\mathfrak{M}}
\def\N{\mathfrak{N}}
\def\X{\mathfrak{X}}

\def\u{\mathfrak{u}}
\def\remove#1{}
\def\diag#1{\text{diag}(#1)}

\newtheorem{theorem}{Theorem}
\newtheorem{lemma}{Lemma}
\newtheorem{proposition}{Proposition}
\newtheorem{corollary}{Corollary}
\newtheorem{remark}{Remark}
\newtheorem{definition}{Definition}

\title[Totally geodesic submanifolds]{A perspective on totally geodesic submanifolds of the symmetric space     $G_2/{\mathop{\hbox{\rm SO}}}(4)$}

\author[Draper C.]{Cristina Draper} 
\author[Mart\'in  C.]{C\'andido Mart\'in Gonz\'alez} 
\subjclass[2010]{Primary  
53C35; %: Symmetric spaces
Secondary 
53C40; %: Global submanifolds
17B25. %: Exceptional (super)algebras	   Primary 53C35, Secondary 57S20, 53C40 pone Kollross
}
\keywords{Totally geodesic submanifolds, Lie triple systems, $G_2$, cross products, associative subalgebras}
\thanks{Supported       by the Spanish Ministerio de Ciencia e Innovaci\'on   through project  PID2023-152673GB-I00    with FEDER funds; and by
Junta de Andaluc\'{\i}a  through projects  FQM-336. The first author is also supported by   PID2020-118452GB-I00. }

\begin{document}
\maketitle
\begin{abstract}
We provide an independent proof of the classification of the maximal totally geodesic submanifolds of the symmetric spaces   $G_2$   and    $G_2/\SO(4)$, jointly with very natural descriptions of all of these submanifolds. 
  The  description of the totally geodesic submanifolds of $G_2$ is in terms of (1) principal subalgebras of $\g_2$; (2) stabilizers of nonzero points of $\R^7$; (3) stabilizers of  associative subalgebras; (4) the set of order two elements in $G_2$ (and its translations). The space $G_2/\SO(4)$ is identified with the   set of associative subalgebras of $\R^7$ 
and its maximal totally geodesic submanifolds can be described as the associative subalgebras adapted to a fixed principal subalgebra, the associative subalgebras orthogonal to a fixed nonzero vector, the associative subalgebras  containing a fixed nonzero vector, and the associative subalgebras intersecting both a fixed associative subalgebra and its orthogonal. A second description is included in terms of Grassmannians, the advantage of which is that the associated Lie triple systems are easily described in matrix form. 
\end{abstract}

%%%%%%%%%%%%%%%%%%%%%%%%%%%%%%%%%%%%%%%%%%%%%%%%%%%%%%%
\section{Introduction}

A submanifold $N$ of a Riemannian manifold $M$ is called totally geodesic
if every geodesic of $N$ is also a geodesic of $M$. 
The classification of the totally geodesic submanifolds in Riemannian symmetric spaces is a relevant problem 
  which has been addressed in concrete ambient spaces. The reference \cite{Chen} contains an illuminating  survey on the topic.   
The problem has been completely solved in the case of the symmetric space $G_2/\SO(4)$. Although some of these submanifolds are missing in the classification by Chen and Nagano in \cite{ChenyNagano}, the classification obtained by Sebastian Klein in \cite{Klein} is complete, based on the reconstruction of the curvature tensor from the Satake diagram previously considered in \cite{KleinSatake}. Both \cite{ChenyNagano} and \cite{Klein} are devoted to symmetric spaces of rank 2. We are interested in exceptional symmetric spaces, and, in this topic, the recent reference \cite{Kollross}  is relevant and of course it encloses    $G_2/\SO(4)$ as a particular case. So the question seems to be what yet another work on the same topic brings to the table. The main aim of this manuscript is to provide \emph{natural} (qualitative)   explicit descriptions of, at least, the maximal totally geodesic  submanifolds of $G_2/\SO(4)$, that is, not
strictly contained in another connected totally geodesic submanifold.
We will   draw our inspiration from the   existing classifications mentioned above to provide geometric descriptions that allow a better understanding of these submanifolds. Equivalently, we aim to give natural descriptions of the related Lie subtriples of the Lie triple system with standard envelope of type $\g_2$. Along the way, we will describe the maximal totally geodesic  submanifolds of the compact Lie group $G_2$ too.
This work is part of an ongoing project on lesser-known symmetric spaces that are quotients of F4, using tensors, forms, and non-associative algebras, all root-free descriptions.
\medskip

Next we summarize more precisely some of the above mentioned  relevant literature on the topic. 
In the summary \cite{Klein09}  by S.~Klein of his own  paper \cite{Klein},   maximal  totally geodesic submanifolds of both   $G_2$ and   the symmetric space $G_2/\SO(4)$ appear. The author tries to be explicit, 
giving   isometric embeddings for the various congruence classes of totally geodesic submanifolds, or at least
 by describing the tangent spaces of the totally geodesic submanifolds. The information in \cite[\S2.9,\,\S2.12]{Klein09},
    concerning $G_2/\SO(4)$ and $G_2$, is reproduced in Table~\ref{tab:Klein}.  
\begin{table}[]
    \centering
    \begin{tabular}{cl||lc}
&$G_2/\SO(4)$&$G_2 $&\vspace{2pt}\\
\hline\vspace{-8pt}\\
5&$\SU(3)/\SO(3)_{\mathrm{srr}=\sqrt{3}}$&$G_2/\SO(4)_{\mathrm{srr}=1}$&8\\
4&$(\mathbb{S}^2_{r=1}\t\mathbb{S}^2_{r=1/\sqrt{3}})/\Z_2$&$(\mathbb{S}^{{3}}_{r=1}\t\mathbb{S}^3_{r=1/\sqrt{3}})/\Z_2$&6\\
4&$\C P^2_{\kappa=3/4}$&$\SU(3)_{\mathrm{srr}=\sqrt{3}}$&8\\
2&$\mathbb{S}^2_{r=\frac{2\sqrt{21}}3}$&$\mathbb{S}^3_{r=\frac{2\sqrt{21}}3}$&3\\
\hline\\ 
\end{tabular}
   \vspace{-2pt} \caption{Maximal t.g.s. of $G_2/\SO(4)$ and $G_2$ \cite{Klein09}}
    \label{tab:Klein}
\end{table}
In order to give information on the isometry type, the notation is as follows: $\mathbb{S}^n_r$ denotes the $n$-dimensional sphere of radius $r$, 
$\C P^2_{\kappa }$ is the complex projective space of constant holomorphic sectional
curvature $4\kappa$,
and the subscript $_{\mathrm{srr}=a}$ is used for describing  the metric by stating the length $a$ of the shortest restricted root of the space.
S.~Klein reconstructs completely the curvature tensor starting from  the Satake diagram of the Riemannian symmetric space under study,  because for him, \emph{in exceptional symmetric spaces,   the explicit description of the
exceptional Lie algebra $\g$ as a matrix algebra is too unwieldy to be useful.} Note that he does it with the help of Maple.

A second work with strong results on the topic  is \cite{Kollross}, where Kollross and Rodr\'\i guez-V\'azquez classify all maximal totally geodesic submanifolds in exceptional symmetric
spaces up to isometry. The main tool is an invariant called the \emph{Dynkin index} of a  totally geodesic
embedding  of semisimple symmetric spaces.  
This manuscript is quite complete, but it does not include the natural descriptions we are looking for. 
Table~\ref{tab:gallegos} is exactly the one appearing in \cite[Table~5]{Kollross} 
on the maximal totally geodesic submanifolds of symmetric spaces of $G_2$-type. 
 A totally geodesic submanifold  is  reflective if it is the connected component of the fixed point set of  an involutive isometry.
The authors work in the non-compact setting due to the fact that duality preserves totally geodesic submanifolds. This makes the appearance of the two tables different. In fact, we have included the dimensions in Table~\ref{tab:Klein} to facilitate comparison. Compare too with our results in Proposition~\ref{pr_tgm_casoG2} and Remark~\ref{re_loscocientesquesalen} (closer to Table~\ref{tab:Klein}, since our context is compact too).
There is some confusion with the presumed \emph{sphere} of dimension 3 in Table 1 (it is missing in \cite{ChenyNagano}, and moreover,   there appear two maximal totally geodesic submanifolds of $G_2$ of dimension 3 
in the summary of \cite{Klein}, which is inconsistent with \cite{Klein09}), so we devote some words in Remark~\ref{re_SOnoesfera} to the clarification of this issue.

\begin{table}[]
    \centering
\begin{tabular}{lllll}
$M$&$\Sigma$&Dynkin index&Reflective?&$\dim\Sigma$\vspace{2pt}\\
\hline\vspace{-7pt}\\
$G^2_2/\SO_4$&$\SL_2(\mathbb R)/\SO_2\t\SL_2(\mathbb R)/\SO_2$&$(3,1)$&Yes&$4$\\
&$\SL_3(\R)/\SO_3$&1&No&$5$\\
&$\SU_{1,2}/\mathrm{S}(\mathrm{U}_1\t\mathrm{U}_2)$&1&No&$4$\\
&$\SL_2(\R)/\SO_2$&28&No&$2$\vspace{2pt}\\
\hline\vspace{-7pt}\\
$G_2(\C)/G_2$&$\SL_2(\C)/ \SU_2\t  \SL_2(\C)/ \SU_2$&$(3,1)$&Yes&$6$\\
&$\SL_2(\C)/ \SU_2$&$28$&No&$3$\\
&$\SL_3(\C)/ \SU_3$&$1$&No&$8$\\
&$G_2^2/\SO_4$&$1$&Yes&$8$\\
\hline\\
\end{tabular}
\caption{Maximal t.g.s. of symmetric spaces of $G_2$-type \cite{Kollross}}
    \label{tab:gallegos}
\end{table}
 \smallskip

This document is structured as follows. Section~\ref{se_prelim} contains some preliminary results on Lie triple systems; totally geodesic submanifolds of a reductive semi-Riemannian homogeneous space;   the Lie group $G_2$ as the stabilizer of a generic 3-form (and its relation with cross products); and the 8-dimensional exceptional symmetric space $\M_S$. We would like to apologize because it is clear that some of these topics are obvious to some researchers, and the rest to others: the reason for including them all together is just so that researchers in both areas, algebra and geometry, can read us. 
We work in Section~\ref{se_stg_G2}  with the description of $G_2$ as the stabilizer of the 3-form to review the results on this group, Proposition~\ref{pr_tgm_casoG2} giving the maximal totally geodesic submanifolds  related to  the maximal Lie triple subsystems described in Proposition~\ref{pr_LTS_G2}.
 This   is the best known case, but it is useful to study it carefully 
  in order to have the tools to deal with  $\M_S$; specially we need some background on principal subalgebras of $\g_2$, related with some totally geodesic submanifolds of both $G_2$ and $\M_S$. (More precisely, related with the lost submanifolds in \cite{ChenyNagano} recovered in \cite{Klein}.)
In Section~\ref{se_stg_MS}, $\M_S$ is considered as the set of associative subalgebras of $(\R^7,\t)$, and its totally geodesic submanifolds  are described in these terms in  Theorem~\ref{pr_tgm_MSv1}, based on the description of the Lie triple subsystems in 
Theorem~\ref{pr_LTS_MSv1}. Note that we describe not only the maximal sub-LTS \emph{up to conjugation}, but all of them. The description as homogeneous spaces becomes very natural too, in Remark~\ref{re_loscocientesquesalen}.
Finally, Section~\ref{se_stg_MS_free} works with $\M_S$ as a submanifold of $\grass\nolimits_3(\R^7)$, denoted by $\M'_S$ to avoid confusions. We describe the Lie triple system with envelope $\g_2$ (the tangent space to $\M'_S$) as the vector space $\sl_3(\R)$ with a convenient triple product in Proposition~\ref{pr_LTSconenvolventeg2}. This allows us to provide descriptions of the maximal subtriples
in matricial terms. The totally geodesic submanifolds  in terms of selfadjoint projections  appear in Corollary~\ref{co_v2}.

Observe that our results provide an independent proof of the classification of the maximal totally geodesic submanifolds  of both $G_2$ and the $8$-dimensional symmetric space.

All through this work the word space will mean finite-dimensional space and of course the same applies to algebras and triple systems.

%%%%%%%%%%%%%%%%%%%%%%%%%%%%%%%%%%%%%%%%%%%%%%%%%%%%%%%%

\section{Preliminary results}\label{se_prelim}

\subsection{Lie triple systems  }\label{se_basicosonLTS}

The day that Descartes had the illuminating idea of assigning a couple of numbers to a point in the plane\footnote{It is said that he was sick in bed and followed the evolution of a fly in the ceiling. 
One of the corners of the ceiling could be used as a reference point.}, he inaugurated a new era in mathematics. He fired the starting
shot of a race consisting in assigning algebraic objects to geometric ones. The idea that the allegedly simple algebraic object could shed light into the geometric one is of course the main motivation for this "algebraization" task. In our context the algebraic structure that we can attach to a totally geodesic submanifold of a Riemaniann symmetric space $M$, is that of a Lie triple subsystem of the Lie triple system associated to $M$ (\cite{Sagle}).
In order to give further details on these ideas we need to define some notions.

An {\em abstract Lie triple system} (see \cite{yamaguti1957algebras}) is defined to be a vector
space $V$ over a field $K$  with an operation  
$V \times V \times V
\to V$ such that $(X,Y,Z)\mapsto [X,Y,Z]$ satisfying
\begin{description}
\item{(i)} $[\cdot,\cdot,\cdot]$ is trilinear over $K$;
\item{(ii)}  $[X,Y,Z]=-[Y,X,Z]$;
\item{(iii)}  $[X,Y,Z] + [Y,Z,X] + [Z,X,Y]=0$,
for any $X,Y,Z\in V$. 
\end{description} An abstract Lie triple system satisfying the additional identity 
\begin{description}
    \item{(iv)} $[X,Y,[A,B,C]]=[[X,Y,A],B,C]+[A,[X,Y,B],C]+[A,B,[X,Y,C]$,
\end{description}
is called a  {\em Lie triple system} (see \cite[Proposition 2.1]{yamaguti1957algebras} and \cite{Lister}).

If $\{e_i\}$ is a basis for an abstract Lie triple system $V$ over the field $K$, then for any choice of the subscripts we have $[e_i,e_j,e_k]:=\sum_{i,j,k,l}c_{ijkl}e_l$, for some scalars
$c_{ijkl}\in K$ which are called the {\em structure constants} of $V$.
An example of an abstract Lie triple system which is not a Lie triple system is the two-dimensional $K$-vector space $K^2$ with
triple product given by the structure constants 
$c_{1211}=1$, $c_{1212}=c_{1221}=c_{1222}=0$.

Any Lie algebra $L$ gives rise to a Lie triple system by defining $[X,Y,Z]:=[[X,Y],Z]$ for any $X,Y,Z\in L$. This triple system will be called the Lie triple system of the algebra $L$. 
 Another way to produce a Lie triple system (over fields of characteristic other than $2$) is to consider a Lie algebra $L$ and an automorphism $\phi\colon L\to L$ with $\phi^2=1_L$. Then $L=\hbox{Fix}(\phi)\oplus \hbox{Antifix}(\phi)$   is a $\mathbb{Z}_2$-grading on $L$,  for     $\hbox{Antifix}(\phi):=\{x\in L\colon \phi(x)=-x\}$. In this case $T=\hbox{Antifix}(\phi)$ is a Lie triple system with the triple product $[[X,Y],Z]$ again.  An \emph{ideal} $I$ of a Lie triple system $T$ is a subspace such that $[I,T,T], [T,I,T], [T,T,I]\subseteq I$. A Lie triple system $T$ is said to be \emph{simple} if $[T,T,T]\ne 0$ and its unique ideals are $0$ and $T$.  
As one learns from \cite{Lister},  the following comprise all simple Lie triple systems (finite-dimensional over algebraically closed fields of characteristic zero):
\begin{itemize}
\item[ (i)] the L.t.s. of simple Lie algebras; 
\item[(ii)] the L.t.s. of elements skew-symmetric relative to automorphisms of order-two in simple Lie algebras   (equivalently, the odd parts of $\mathbb{Z}_2$-graded simple Lie algebras).     
\end{itemize}
Thus, Lister reduces the classification problem to the determination of order-two automorphisms of simple Lie algebras. 

Recall that a derivation $d$ of a Lie triple system $(T,[\cdot,\cdot,\cdot])$ over a field $K$ is a $K$-linear map $d\colon T\to T$ such that for any $x,y,z\in T$ one has $d([x,y,z])=[d(x),y,z]+[x,d(y),z]+[x,y,d(z)]$. The space of derivations of $T$, denoted $\hbox{Der}(T)$, is a Lie algebra for the usual Lie bracket. The maps $L_{x,y}\colon T\to T$ such that $L_{x,y}(z):=[x,y,z]$ (where $x,y,z\in T$) are derivations of $T$, due to (iv). The Lie subalgebra  of $\hbox{Der}(T)$ generated by all the $L_{x,y}$'s,
is usually denoted by $\hbox{InnDer}(T)$ and its elements are termed {\em inner derivations}.
In fact $\hbox{InnDer}(T)$ is an ideal of $\hbox{Der}(T)$. Furthermore,  we can construct a Lie algebra 
$$\g(T):=\hbox{InnDer}(T)\oplus T$$
by defining a product $[(f,x),(g,y)]:=([f,g]+L_{x,y},f(y)-g(x))$ for any $f,g\in\hbox{InnDer}(T)$, $x,y\in T$ 
(see \cite[Theorem 2.1]{yamaguti1957algebras}). The Lie algebra $\g(T)$ is called the {\em envelope}, or standard envelope,  of $T$. It is a $\Z_2$-graded Lie algebra with even part $\hbox{InnDer}(T)$ and odd part $T$. Thus we can speak of the $\Z_2$-graded envelope $\g(T)$ to be more explicit. Notice that when we take $x,y,z\in T$, we have $[x,y,z]=[[x,y],z]$ where the binary Lie bracket $[\cdot,\cdot ]$ is that of $\g(T)$.  In other words, the property (iv) makes that any $T$ can be one-to-one embedded into a Lie algebra $\g(T)$ such that $T$ is subtriple of it.
\smallskip

  An {\em associative triple system} $A$ is a vector space equipped with a trilinear 
map $A\times A\times A\to A$ such that 
$(x,y,z)\mapsto \tp{xyz}$ satisfies  $\tp{\tp{xyz}uv}=\tp{x\tp{uzy}v}=\tp{xy\tp{zuv}}$
for arbitrary elements $x,y,z,u,v\in A$. 
If $A$ is an associative triple system relative to $\esc{xyz}$, 
then $A^-$ (equal to $A$ as a vector space) endowed with the triple product $[x,y,z]$ defined as 
$$
[x,y,z]:=  \esc{xyz}-\esc{yxz}-\esc{zxy} + \esc{zyx},
$$
 is a Lie triple system that we will call the {\em  skew-symmetrization} of $A$.
 Thus, the associative triple system of rectangular matrices $\mat_{n,m}(\R)$ with triple product $\tp{xyz}=xy^tz$ (where $y^t$ stands for the transposition of the matrix $y$) gives rise to the Lie triple system $\mat_{n,m}(\R)^-$. Following \cite{Lister}, this Lie triple system is simple, with standard envelope isomorphic to the orthogonal algebra $\mathfrak{so}(n+m)$.
 Also, if $A$ is a triple system with product $\esc{\cdots}$, then $A_*$ denotes the so called {\em twin triple system} whose product is the opposite $-\esc{\cdots}$. Dealing with triple systems over the reals, we do not have $A\cong A_*$ (for complex triple systems we do have $A\cong A_*$).

%%%%%%%%%%%%%%%%%%%%%%%%%%%%%%%%%%

\subsection{   Curvature and triple systems}\label{se_curvatura}

As one learns for instance from \cite{Sagle}, if $M$ is 
a $C^\infty$-manifold and  $D$ is a covariant differentiation operator (that is, a connection) defined on $M$, then for each pair  $X, Y$  of $C^\infty$-vector fields defined on  $M$, there is  a 
$C^\infty$-vector field $D_XY$. Then we can define the torsion  tensor of $D$ by
$   T (X, Y) = D_XY - D_YX - [X, Y]$.
It  satisfies $T(X,Y)=-T(Y,X)$ and it is $C^\infty$-linear.
One can also define the
curvature tensor, $R$, for $C^\infty$-vector fields $X$, $Y$ and $Z$  by
$$R(X,Y)Z=[D_X,D_Y]Z-D_{[X,Y]}Z.$$
Thus the   curvature measures the failure of the map $X\mapsto D_X$ to be a Lie algebra homomorphism.
If $T(X,Y)=0$ for any two vector fields, then the curvature tensor $R$ satisfies the first Bianchi identity or algebraic Bianchi identity:
$$
R(X,Y)Z + R(Y,Z)X + R(Z,X)Y = 0.
$$ 
Summarizing the above discussion, if $T(X,Y)=0$ for any $X,Y$, then the space $\mathfrak{X}( M )$ of all vector fields on $M$
is an   abstract   Lie triple system   relative to $[X,Y,Z]:=R(X,Y)Z$. 
This is precisely the case 
for a semi-Riemannian manifold $M$ and the Riemann curvature tensor (\cite[Chapter I, Section 2, p.~28]{Loos})  
$R\colon \X ( M ) \times \X ( M ) \times \X ( M ) \to \X ( M )$ defined by 
$R ( X , Y ) =  [ \nabla_X , \nabla_Y ]  - \nabla_{[ X , Y ]}$
 for  $\nabla$   the Levi-Civita connection. 
 That is, $\X(M)$ is an   abstract
Lie triple system relative to the operation 
\begin{equation}\label{eq_eltripledeR}
[X,Y,Z] := R(X,Y)Z.
\end{equation}

 Furthermore, following Sagle in \cite[Section 3]{Sagle}, if    $M'$ is 
 a  totally geodesic nonsingular submanifold of the  semi-Riemannian manifold $M$ 
 (so the metric tensor of $M$ when restricted to the tangent space $T_pM'$ for all $p\in M'$ is nondegenerate on $T_pM'$),
 then if $X,Y,Z$ are in $\X(M')$, one has $[X,Y,Z] = R(X,Y)Z\in \X(M')$, so that 
  $(\X(M'),[\cdot,\cdot,\cdot])$ is also an  abstract   Lie triple subsystem of $\X(M)$. 
   A converse is also proved in \cite[Theorem, p.\,11-12]{Sagle} in the setting of connected reductive homogeneous spaces. 
   More precisely, assume that    $M=G/H$,  for $G$   a connected Lie group and $H$ a closed connected Lie subgroup, 
   where the tensor metric is of course $G$-invariant, 
   and that
 there exists a subspace $\m$ of $\g$ such that $\g = \h\oplus\m$  and  $[\h,\m]\subset \m$,
 being  $\g$ and $\h$ the Lie algebras of the Lie groups $G$ and $H$. 
  Then, for any nonsingular subspace $\m'\subset\m$   closed for the triple product~\eqref{eq_eltripledeR}  such that $[\m',\m']\subset\h+\m'$,
  there exists  a totally geodesic nonsingular submanifold $M'$ of $M=G/H$ containing $p_0=H$ such that  $T_{p_0}M'= \m'$.  
  In conclusion, there is a   correspondence
  between totally geodesic   submanifolds of semi-Riemannian reductive homogeneous spaces
  and certain (abstract) subtriples of  their tangent spaces.  
 \smallskip

   The link between totally geodesic submanifolds and Lie triple subsystems is easier to materialize in the setting of Riemmanian symmetric spaces.  
   Note first that, if   $M$ is a locally symmetric space, i.e. the covariant derivative of the Riemann curvature tensor is $0$, then the product in Eq.~\eqref{eq_eltripledeR} satisfies (iv) too  (\cite[p.~110]{yamaguti1957algebras}).    This means that  $\X(M)$, with the triple product given by the curvature, is not only an abstract triple system but   actually 
   a Lie triple system (see also \cite[Proposition~2.1, Theorem~2.1]{yamaguti1957algebras}).  
   Recall that a Riemannian manifold  $M$ is 
a {\em symmetric space} if, for
every point $p\in M$ there exists an isometry $s_p$ (called {\em geodesic reflection}) such that  $s_p(p)=p$
and $s_{*p} =-\id_{T_pM}$.  
If $M$ is  connected, and $G=\mathrm{Iso}(M)$ is its group of isometries, the geodesic symmetry at a point $o$ gives raise to
an involutive automorphism $\sigma\colon G\to G$ ($g\mapsto s_ogs_o$) whose differential induces an order-two automorphism $\phi=(d\sigma)_e$ of the Lie algebra $\g=T_eG$ of $G$, called the Cartan involution. 
By already mentioned standard arguments, $\phi$ provides a $\mathbb{Z}_2$-grading $\g=\h\oplus\m$ 
with even part $\h=\hbox{Fix}(\phi)$ and odd part  $\m=\hbox{Antifix}(\phi)$ (eigenspace of eigenvalue $-1$). 
Moreover, the curvature tensor of $M$ at $o$ can be expressed as $R_o(X,Y,Z)=-[[X,Y],Z]$.
It follows that the Lie triple system $\m$ inherited by the one of $\g$ coincides with the twin of the Lie triple system given by the curvature. 
A Lie triple system and its twin are not in general isomorphic, but notice that   there is no general agreement about the sign of the curvature tensor in the literature. 
As $M$ is a reductive homogeneous space, we can apply the above results to get, if $S$ is a totally geodesic submanifold of $M$ and $p_0\in S$, that the subspace $\mathfrak{s}:=T_{p_0}S$ of $\m$ is a Lie triple subsystem   (see a direct argument for instance in \cite[Ch.~IV, Theorem~7.2]{Helgason}).
It is also proved in the aforementioned reference that if we take a Lie triple subsystem $\mathfrak{s}$ of $\m$ and define $S=\exp(\mathfrak{s})$, then $S$ has a natural differentiable structure relative to which it is totally geodesic submanifold of $M$ such that $T_{p_0}S=\mathfrak{s}$. This in particular explains  the  abundance of totally geodesic submanifolds in symmetric spaces.

From the previous paragraph one can realize that the evolution of ideas since the original coordinatization of flies in Descartes' ceiling room, until the described Lie triple systems 
viewpoint in the classification of totally geodesic submanifolds of symmetric Riemannian manifolds, is ample.
Taking advantage of this approach, some authors have  
performed thoughtful and meritoriously detailed classifications of totally geodesic submanifolds in different settings. 
The authors of \cite{Kollross} summarize some of the existing classifications and moreover they classify totally geodesic submanifolds in exceptional symmetric spaces.

%%%%%%%%%%%%%%%%%%%%%%%%%%%%%%%%%%%%%%%%%%%%%%%%%%%%%%%%

\subsection{Cross products and $G_2$}

It is well-known that $\R^7$ can be endowed with a {\em cross-product}, that is, an anticommutative  bilinear map $\times\colon\R^7\times\R^7\to\R^7$ such that $\esc{x\t y,x}=0=\esc{x\times y,y}$ and 
$$
\esc{x\times y,x\times y}=\det \begin{pmatrix}
 \esc{x,x}&\esc{x,y}\\\esc{y,x}&\esc{y,y}   
\end{pmatrix}, 
$$ where $\esc{\cdot,\cdot }$ denotes the usual scalar product in $\R^7$.
(Nonzero cross products exist only in dimensions $3$ and $7$, see \cite{CrossProducts}.)

For instance, if $\{e_i:i=1\dots 7\}$ denotes the canonical basis of $\R^7$, one could consider the cross product given by the anticommutativity and
\begin{equation}\label{uncross}
  e_{i} \times e_{i+1}=e_{i+3},\quad
e_{i+1} \times e_{i+3}=e_{i},\quad
e_{i+3} \times e_{i}=e_{i+1},\quad  
\end{equation}
for the sum of indices taken modulo 7.
 
A cross product satisfies 
 (see \cite[p.\,1]{Kobayashi}) 
\begin{equation}\label{props_cross}
(x\t y)\t z+x\t(y\t z)=2\esc{x,z}y-\esc{y,z}x-\esc{x,y}z,
\end{equation}
and $\esc{x\times y,z}=\esc{x,y\times z}$. 
In particular, we can consider the alternate trilinear map $\Omega\colon\R^7\times\R^7\times\R^7\to \R$ given by 
$$\Omega(x,y,z):=\esc{x\t y,z}.$$ 
Note that  $\Omega$ can be identified with a 3-form $\wedge^3\R^7\to\R$
and 
consider the canonical action $\GL(7,\R)\times (\wedge^3\R^7)^*\to (\wedge^3\R^7)^*$ given as usual by  
$
F\cdot \omega=\omega ( F^{-1}\o F^{-1}\o F^{-1}).
$
The Lie group $G_2$ can be defined as the stabilizer $G_\Omega=\{F\in \GL(7,\R):\Omega=F\cdot \Omega\}$, 
and its tangent Lie algebra consists of the elements
$d\in\mathfrak{gl}(7,\R)$ such that
\begin{equation}\label{eq_glomega}
\Omega(d(x),y,z)+\Omega(x,d(y),z)+\Omega(x,y,d(z))=0.
\end{equation}
As 
the elements in $G_\Omega$ are isometries, then the linear maps tangent to it belong to $\mathfrak{so}(\R^7,\esc{\cdot,\cdot })\equiv \mathfrak{so}(7)$.
Hence the above condition is equivalent to
$$
d(x\t y)=d(x)\t y+ x\t d(y),
$$
for arbitrary $x,y\in\R^7$, taking into consideration that $\esc{\cdot,\cdot }$ is nondegenerate. In other words, $\mathfrak{g}_2$ can be identified both with $\mathfrak{gl}(7,\Omega)$ (defined by \eqref{eq_glomega}) and   with $\der(\R^7,\t)$, which is a Lie subalgebra of $ \mathfrak{gl}(7,\R)$.

One can introduce the cross product $\times$ in $\R^7$ in a coordinate free fashion. 
Recall that a {\em generic} 3-form is defined to be an element in 
$(\wedge^3\R^7)^*$ whose orbit under the $\GL(7,\R)$-action above is open.  In 1900, F.~Engel 
proved that there are just two orbits of generic 3-forms, while only one over $\mathbb C$ (read in \cite{Ilka} some nice details on this story).  
Now, as mentioned in \cite[(2.2) and Lemma 3.9]{anderson2011chern}, any trilinear form $\gamma\colon \wedge^3\R^7\to\R$ induces a symmetric bilinear form $\b_\gamma$\ by composing $(u,v)\mapsto -\frac{1}{3}\gamma(u,\cdot,\cdot)\wedge\gamma(v,\cdot,\cdot)\wedge\gamma$ with an isomorphism $\wedge^7(\R^7)^*\cong\R$
(see, alternatively,  \cite[Lemma~4.10]{MiG2}).
In any case, out of the two orbits of  generic $3$-forms, there is only one whose associated bilinear form is definite. Furthermore if we denote by $\omega$ one representative of this orbit, there is a cross product $\times_\omega$ such that $\beta_\omega(x\times_\omega y, z)=\omega(x,y,z)$.  
This cross product is precisely the one we are dealing with at the beginning of the section. It is essentially unique, though not unique in absolute terms. 

 For any generic 3-form $\omega$, the isotropy group $G_\omega=\{F\in \GL(7,\R):\omega=F\cdot \omega\}$
   has dimension 14 and it is besides simple.  F.~Engel proved that the respective isotropy groups of the two aforementioned orbits are isomorphic to $G_2$ and to 
the   noncompact real form    $G_2^2\subset \SO(4,3)$.

\begin{remark}\label{re_octonions}
{\rm
   We can recall here the close relationship between cross products and octonion algebras. Fixed a nonzero cross product $(\R^7,\t)$, take $\O=\R 1\oplus\R^7$ with product $xy=-\esc{x,y} 1+x\t y$ if $x,y\in\R^7$ and where $1$ is the unit.
   Then $(\R^7,\t)$ is a Malcev algebra,
    and $\O$ is an octonion algebra, that is,
   a unital $8$-dimensional algebra  endowed with a  nondegenerate quadratic form $n\colon\O\to\R$ ($n(x)=\esc{x,x}$), which is multiplicative, that is, $n(xy)=n(x)n(y)$. The two orbits of generic 3-forms (equivalently, of cross products) give two classes of nonisomorphic octonion algebras, namely, the usual octonion division algebra, related to the definite norm, and the so called split octonion algebra, whose related quadratic form has neutral signature. Conversely, if $\O$ is an octonion algebra, it is necessarily a quadratic algebra, so that every element $x\in\O$ satisfies a degree 2 equation $x^2-t(x)x+n(x)1=0$, for $t\colon\O\to\R$ a linear map called {\em trace}. Then the set of zero trace elements $\O_0$ is a copy of $\R^7$ and the product $x\t y=xy-\frac12t(xy)1$ provides a cross product in $\O_0$.
    In what follows, $\O$ will denote the division algebra obtained by adding the unit to $(\R^7,\t)$, for $\t$ the concrete cross product in \eqref{uncross}.
   
   From the above we know that 
     passing through octonions   is not completely necessary   to work with the compact Lie group $G_2$, 
     we can work directly with a generic 3-form or with its related cross product. (The same can be applied to the symmetric space $G_2/\SO(4)$.)  
}\end{remark}

%%%%%%%%%%%%%%%%%%%%%%%%%%%%%%%%%%%%%%%%%%%%%%%%%%%%%%%%
%%%%%%%%%%%%%%%%%%%%%%%%%%%%%%%%%%%%%%%%%%%%%%%%%%%%%%%%

\subsection{The exceptional 8-dimensional symmetric space  $G_2/\SO(4)$}\label{sub_M_S}

 Assume we have fixed a cross product, as that one in \eqref{uncross}.
A usual description of the 8-dimensional symmetric space is that of the set of {\em associative subalgebras} of $\R^7$, 
which means 
$\M_S=\{V\le\R^7:\dim V=3,\,V\t V\subset V\}$. 
Each  tridimensional  $V\le\R^7$ with $V\times V\subset V$ induces an associative algebra $\R\oplus  V$ for 
the usual product $(\l,v)(\l',v'):=(\l\l'-\esc{v,v'}, \l v'+\l'v+v\times v')$. Technically speaking, $(V,\t)$ is not an associative algebra but it induces one and that is why we use the terminology \lq\lq associative subalgebra". It may be assumed as a terminology abuse despite its meaning is clear. 
For  instance, for any $i=1,\dots,7$ the subspace $V^i=\langle\{ e_i,e_{i+1},e_{i+3}\}\rangle$ is such an associative subalgebra.

As $G_2=G_\Omega$ preserves the cross product, it acts on $\M_S$ and moreover, the action is transitive  (\cite[Corollary 2.2.4]{Springer}). In fact, for any other associative subalgebra $V$, take $\{\iu,\ju,\iu\t\ju\}$ an orthonormal basis of $V$ and a unit vector $\ell\in V^\perp$, and then there exists a  unique $F\in G_2$ such that $F(e_1)=\iu$, $F(e_2)=\ju$ and $F(e_3)=\ell$. 
In particular any $V\in\M_S$ determines a $\Z_2$-grading on the algebra
$(\R^7,\t)$, because not only $V\t V\subset V$ is fulfilled, but also
\begin{equation}\label{eq_Z2grad}
    V\t V^\bot\subset V^\bot,\qquad  V^\bot\t V^\bot\subset V,
\end{equation}
since the same is true for $V^1=\esc{\{e_1,e_2,e_4\}}$ and $(V^1)^\perp=\esc{\{e_3,e_5,e_6,e_7\}}$, in view of  \eqref{uncross}.
Thus, $V$ and $V^\bot$ are, respectively, the even and odd part of the graded algebra
$\R^7=V\oplus V^\bot.$ (Equation \ref{eq_Z2grad} can be alternatively deduced by applying carefully the formula in \eqref{props_cross} to different triples of elements.)
In relation with Remark~\ref{re_octonions}, $V\in\M_S$ if and only if the associative algebra $Q=\R\oplus  V$ is a quaternion subalgebra of $\O=\R 1\oplus\R^7$, which is $\Z_2$-graded with even part $\O_{\bar0}=Q$ and odd part $\O_{\bar1}=Q^\bot=V^\bot$.

The isotropy group of any element $V\in\M_S$ by the $G_2$-action can be identified with $\SO(V^\perp)\cong\SO(4)$ by $F\mapsto F\vert_{V^\perp}$  (\cite{Akbulut,procAlb}). 
In fact, the elements in the group $G_2$ are isometries so that any $F$ leaving invariant $V$ also leaves invariant the orthogonal subspace, but   $V^\bot$ generates $\R^7$ as an algebra, so that the  restriction $F\vert_{V^\perp}$ allows us to recover all the information of $F$. Thus $\M_S$   is identified with the homogeneous space $G_2/\SO(4)$, which in particular  endows $\M_S$ with a manifold structure  (see, alternatively, Remark~\ref{re_subalgebras}).

In order to describe this manifold in terms of the 3-form, let us characterize when a $3$-dimensional subspace  $V $ of $\R^7$ belongs to $\M_S$. 
Clearly, \eqref{eq_Z2grad} gives the necessary condition   
  $\Omega(V^\bot,V^\bot,V^\bot)=0$   (equivalently, $\Omega(V,V,V^\bot)=0$). According to \cite[Lemma~3.1]{procAlb}, the converse is true, that is, if a 4-dimensional vector subspace $W$
of $\R^7$ satisfies $\Omega(W,W,W)=0$, then $W^\bot\in\M_S$.
Consequently,   $\M_S$ is in bijective correspondence with  
$ \{W\le\R^7:\dim W=4,\,\Omega(W,W,W)=0 \}$ (just the description of the symmetric space considered in \cite{procAlb}),
by means of $V\mapsto V^\bot$. This identification is compatible with the natural action of $G_2=G_\Omega$ on both sets.

 Recall that the
Grassmanian $\grass_n(\R^7)$ as a set  is nothing but the set of all $n$-dimensional 
subspaces of $\R^7$. Each $n$-dimensional subspace $V\subset \R^7$ can be identified with the (orthogonal) projection $\pi\colon\R^7\to\R^7$ onto $V$. But $\pi$ is a self-adjoint idempotent of $\hbox{End}_{\R}(\R^7)$ and has trace $n$. So we could say 
\begin{equation}\label{fat}
\grass\nolimits_n(\R^7)=\{\pi\in \hbox{End}_{\R}(\R^7)\colon \pi^2=\pi, \pi^\sharp=\pi, \tr(\pi)=n\},
\end{equation}
where $\tr$ is the trace of the linear map and $\pi^\sharp$ denotes the adjoint of $\pi$ relative to the Euclidean inner product. If we prefer a matrix representation of the Grassmanian, we could say that it is
the set of $7\times 7$ idempotent symmetric matrices of $\mat_7(\R)$ with trace $n$.  

Notice that the condition $x\in V$ is just $\pi_V(x)=x$ when $\pi_V$ denotes the orthogonal projection onto $V$.
Also $\pi_V(x)=0$ if and only if $x\in V^\bot$.   Thus   $\Omega (\pi_V\o\pi_V\o(1-\pi_V))=0$  and $\Omega (\pi_{V^\bot}\o\pi_{V^\bot}\o\pi_{V^\bot})=0$ 
(the notation $\pi_1\otimes\pi_2\otimes\pi_3$ stands for the map $(\R^7)^{3\otimes}\to (\R^7)^{3\otimes}$ such that $x\o y\otimes z\mapsto \pi_1(x)\o\pi_2(y)\o\pi_3(z)$). Taking into account all of this, we can identify the symmetric space $\M_S$ with $\{\pi\in \grass\nolimits_4(\R^7)\colon \Omega(\pi\o\pi\o\pi)=0\}$ and with
\begin{equation}\label{emese}
\M'_S:=\{\pi\in \grass\nolimits_3(\R^7)\colon \Omega(\pi\o\pi\o\pi')=0\},
\end{equation}
for $\pi'=1-\pi$.
 For any $\pi\in\M_S'$, Equation~\eqref{eq_Z2grad} says that $\R^7=V_{0}\oplus V_{1}$ is a $\Z_2$-grading, for  
 $V_0=\ker(\pi')=\text{Fix}(\pi)$
 and $V_1=\text{Fix}(\pi')=\ker(\pi)$, that is, $V_\lambda$ coincides with the eigenspace of $\pi'$ related to the eigenvalue $\lambda$ for both $\lambda=0,1$ (denote  $V_{0,\pi'}$ and $V_{1,\pi'}$ if there is ambiguity).
 From this perspective, it is again clear how the group $G_2= \{F\in\text{\rm GL}(\R^7)\colon F\cdot\Omega=\Omega\}$ acts on $\M'_S$ by $F\cdot \pi=F\pi F^{-1}$ (take into account that $F^{-1}=F^\sharp$). The mentioned diffeomorphism between $\M_S$ and $\M'_S$ is also compatible with the $G_2$-action.

 \begin{remark}\label{re_subalgebras}
 {\rm 
 The above manifold structure of $\M_S$ is a particular case of the following situation.
     If $\R^n$ is a real algebra relative to a product
$\mu\colon \R^n\otimes\R^n\to\R^n$,   consider the Grassmanian $\grass_k(\R^n)$ of all $k$-dimensional subspaces of $\R^n$. Again it can be identified with  projections  $\pi\colon \R^n\to \R^n$ in the Euclidean space $\R^n$, that is, $\pi^2=\pi$, $\pi^\sharp=\pi$ (self-adjoint), $\tr(\pi)=k$. These equalities translate into algebra equations on the entries of the matrix of $\pi$ relative to any basis.
Thus $\grass_k(\R^n)$ is an algebraic set and hence a manifold. Consider then the subset $\hbox{\bf Subalg}_k(\R^n):=\{\pi\in\grass_{k}(\R^n)\colon \pi'\mu\pi^{\otimes 2}=0\}$ where $\pi^{\otimes 2}:=\pi\otimes\pi$, $\pi':=1-\pi$. Any element $\pi\in\hbox{\bf Subalg}_k(\R^n)$ defines a subalgebra $\hbox{Fix}(\pi)\subset\R^n$, and conversely any subalgebra (whose underlying space is associated to some $\pi$) satisfies
$\pi'\mu\pi^{\o 2}=0$. To prove that $\hbox{\bf Subalg}_k(\R^n)$ is a submanifold of $\grass_k(\R^n)$, take into account that  $\pi'\mu\pi^{\o 2}=0$ gives a system of algebraic equations of degree $3$: fixing a basis $\{e_i\}$ of $\R^n$, if we write $\pi(e_i)=x_i^j e_j$ (following Einstein convention), then $\pi'\mu\pi^{\o 2}=0$ if and only if $x_i^a x_j^b \omega_{a,b}^k(1-x_k^q)=0$ for any $i,j,q$ (the $\omega^k_{a,b}$'s are the structure constants relative to the given basis). So $\hbox{\bf Subalg}_k(\R^n)$ is both an algebraic set and  a manifold.  
}
 \end{remark}

%%%%%%%%%%%%%%%%%%%%%%%%%%%%%%%%%%%%%%%%%%%%%%%%%%%%%%%%
 
\section{Maximal totally geodesic submanifolds of $G_2$} \label{se_stg_G2}

Although the Lie triple subsystems of $\g_2$ are described in 
\cite[Theorem~5.2]{Klein} (not only the maximal ones), at least by means of roots (with the help of Satake diagrams), 
we are interested in introducing descriptions and notations with the focus on the following section on the symmetric space $G_2/\SO(4)$,
as well as with the desire to obtain more natural descriptions of the triples.

\subsection{Some considerations on principal subalgebras}\label{se_ppaldefs}

First, a key to understand the weird cases, missing in \cite{ChenyNagano},  is to recall some facts on principal subalgebras
(not so weird in the end, when one understands that this is a very frequent subalgebra). Recall that, for $\g$ a complex simple Lie algebra, a {\em principal subalgebra} of $\g$ is a three-dimensional simple subalgebra (necessarily isomorphic to $\sl_2(\C)$)
which contains a principal nilpotent element. We will usually shorten its name to TDS.
 An element in a Lie algebra is said {\em nilpotent} if its adjoint map is a nilpotent linear map. 
The set $\{e\in\g: e\  \mathrm{ nilpotent}\}$ is an algebraic variety with only one dense orbit, called the principal orbit. Its elements are  the {\em principal} nilpotent elements.
 The conjugate classes of TDS  in $\g$ are in   one-to-one correspondence with the conjugate classes of non-zero nilpotent elements in $\g$ \cite[Corollary~3.7]{Kostant}.
It happens that an invariant distinguishing the conjugacy classes of TDS  is the so called {\em index}, introduced by Dynkin in the fifties  \cite{Dynkin52}
as a tool to distinguish different (non-conjugate) embeddings of isomorphic subalgebras.
In the case of $\g_2$, the principal subalgebras are those of index 28.

If now $\g$ is a real simple Lie algebra, a three-dimensional subalgebra is called  principal if its complexified algebra so is. 
In the case of the compact $\g_2$, the set of principal subalgebras turns out to be a homogeneous space of type $G_2/\SO(3)^{\textrm{irr}}$ (\cite[\S3.9]{procAlb}). Simply because all the principal subalgebras of the compact Lie algebra $\g_2=\der(\R^7,\t)$
are conjugated and, for $\h$ one of them,  
then the isotropy group
$$
H^\h:=\{F\in \GL(\R^7,\t)=G_2:\mathrm{Ad}F(\h)\subset \h\}
$$
is a 3-dimensional simple group isomorphic to $\SO(3)$ (see Remark~\ref{re_SOnoesfera}). (For any $F\in \GL(n,\R)$, $X\in\mathfrak{gl}(n,\R)$, we denote $\mathrm{Ad}F(X)=FXF^{-1}$.) In particular, if one chooses  a TDS at random, it is most likely to be a principal subalgebra.
As can be concluded for instance from \cite[Theorem~2.2]{procAlb},  
$\h$ is principal if $\R^7$ is an absolutely irreducible $\h$-module. In particular (see Lemma~\ref{le_maximalehr} and Proposition~\ref{prop_todascasoalg} below), principal TDS  are maximal subalgebras and also maximal as LTS. 

\subsection{Lie triple systems of $\g_2$}

Recall that we are considering the Lie algebra $\g_2$, but with its structure of LTS, i.e., $[x,y,z]=[[x,y],z]$. We are interested in the maximal LTS of $\g_2$. 
For the sake of completeness, we will  include in Proposition~\ref{prop_todascasoalg}  a self-contained proof that any maximal Lie triple subsystem of a simple real or complex Lie algebra $\g$ is either a maximal subalgebra or the the odd part of a nontrivial $\Z_2$-grading. Our knowledge of the maximal subalgebras in \cite{Dynkin52} together with the knowledge of the order two automorphisms from \cite[Chapter~8]{Kac}  allows us to apply it in further settings. In the case of $\g_2$, the characters appearing in this way in Proposition~\ref{pr_LTS_G2} will   play a leading role in the subsequent geometrical descriptions.

\begin{lemma}\label{le_maximalehr}
Take $\g=\h\oplus\m $   a reductive decomposition of a real or complex Lie algebra $\g$ (that is, $\h$ is subalgebra and $[\h,\m]\subset\m$).
If $\h$ is semisimple and $\m$ is an irreducible $\h$-module, then $\h$ is a maximal subalgebra  of $\g$. 
\end{lemma} 

\begin{proof}
Assume that $\h\subsetneq\h'\subsetneq \g$ with $[\h',\h']\subset\h'$. As we are working on a characteristic zero field  (the lemma is valid in this more general setting), every $\h$-module is completely reducible. This can be applied to the own $\h$ as a submodule of $\h'$. So there exists an $\h$-module  $\m'\ne0$ such that $\h'=\h\oplus\m'$. Now $\h'$ is an $\h$-submodule of $\g$ and we can take $\m''\ne0$ such that $\g=\h'\oplus\m''$. By uniqueness, $\m\cong\m'\oplus\m''$, which contradicts the irreducibility of $\m$ as $\h$-module.
\end{proof}

 Note that if a subalgebra $\h$  of a Lie algebra $\g$   is  maximal as subtriple of the LTS $\g$, then it is evidently maximal as subalgebra. Under the conditions we are interested in, the converse will be true.
 \begin{proposition}\label{Leo}
Let $\g$ be a   simple Lie algebra,   either complex or real   of dimension greater than 3. 
Then any  maximal subalgebra $\h$  of $\g$   is a maximal subtriple of the LTS $\g$.
\end{proposition}
\begin{proof}
Assume that there is a Lie triple subsystem $T$ of $\g$ with $\h\subsetneq T\subsetneq\g$. If we denote by  $T^2:=[T,T]$, 
then $[T^2,T]\subset T$ and Jacobi identity gives $[T^2,T^2]\subset T^2$, so
we have a subalgebra $T+T^2$ of $\g$ such that $\h\subset T+T^2$. By maximality of $\h$ we have either $\h=T+T^2$ or $T+T^2=\g$. But the first possibility gives $\h=T$ which we have ruled out from the beginning. Thus $\g=T+T^2$.
Notice that $T\cap T^2$ is an ideal of the algebra $\g$: 
$$[T\cap T^2,\g]\subset[T\cap T^2,T]+[T\cap T^2,T^2]$$ and each of the summands is contained in $T\cap T^2$. By simplicity of $\g$ we have either $T\cap T^2=\g$ or $T\cap T^2=0$. But the first possibility implies $\g=T$ which is impossible. Thus we have $T\cap T^2=0$ and $\g=T\oplus T^2$ with $\h\subset T$. Consequently $[\h,\h]\subset T^2\cap T=0$
and $\h$ is abelian. On the other hand the normalizer $N_\g(\h)=\{x\in\g:[x,\h]\subset \h\}$ is a subalgebra containing $\h$. Again we have two possibilities, either $N_\g(\h)=\g$ or $N_\g(\h)=\h$.
In the first case   $\h$ is an ideal of $\g$ but $\g$ is simple. So $\h=0$ or $\h=\g$ and both possibilities are contradictory. If $\h$ is its own normalizer in $\g$, then $\h$ is a Cartan subalgebra of $\g$. But in the complex case a Cartan subalgebra is never a maximal subalgebra, and in the real case a Cartan subalgebra is maximal only if $\g$ is a TDS. 
\end{proof}

\begin{remark}
    Note, for $\{h_1,h_2,h_3\}$ a basis of $\mathfrak{so}(3)$ with $[h_i,h_{i+1}]=h_{i+2}$, that $\h=\esc{h_1}$ is a maximal subalgebra of $\mathfrak{so}(3)$
    (if $[h_1,sh_2+th_3]=\lambda(sh_2+th_3)$, then $\lambda^2=-1$);
    while it is not a maximal subtriple.  
    Furthermore, $ T=\esc{h_1,sh_2+th_3}$ is a  subtriple for all scalars $s,t$ since
    $[h_1,sh_2+th_3,h_1]=sh_2+th_3$ and $[sh_2+th_3,h_1,sh_2+th_3]=(s^2+t^2)h_1$. 
\end{remark}

At some point we will need to know what all the LTS of the algebra $\mathfrak{so}(3)$  look like. 
\begin{lemma}\label{le_LTSapatadas}
    For any $0\ne X,X'\in \mathfrak{so}(3)$, then $\esc{\{X,X'\}}$ is a LTS.
\end{lemma}

\begin{proof}
    The normal form of a skew-symmetric matrix with real coefficients says that any such matrix $A\ne 0$ is of the form
$$  P\,\diag{k_1 S,\ldots,k_n S,0,\ldots,0}\, P^{-1},$$ where $S=\tiny\begin{pmatrix}0 & 1\\-1 & 0\end{pmatrix}$, $k_i\in\R\setminus\{0\}$ and $P$ an orthogonal matrix, $P^{-1}=P^t$. In particular, in $\mathfrak{so}(3)$ any nonzero element is of the form 
$$
P{\tiny \begin{pmatrix}0 & k & 0\\-k & 0 & 0\\0 & 0& 0\end{pmatrix}}P^{-1}
$$
for $k\in\R\setminus\{0\}$ and an orthogonal matrix $P$. So if $0\ne X\in \mathfrak{so}(3)$ then there is a nonzero real $r$ and an orthogonal matrix $P$ such that $rX=P(E_{12}-E_{21})P^{-1}$. Then $Y=P(E_{23}-E_{32})P^{-1}$ and 
$Z=P(E_{31}-E_{13})P^{-1}$ form a triple $\{rX,Y,Z\}\subset \mathfrak{so}(3)$ (contained in $\mathfrak{so}(3)$ because $P$ is an orthogonal matrix)
such that $[rX,Y]=Z$, $[Y,Z]=rX$ and $[Z,rX]=Y$. As in the above remark, the fact that $\esc{\{rX,sY+tZ\}}$ is a LTS easily follows.
\end{proof}

Now we investigate another way to produce maximal Lie triple subsystems.
\begin{lemma}\label{lema2}
    Let $\g$ be a simple Lie algebra over $\R$ or $\C$. For any  nontrivial $\Z_2$-grading $\g=\g_\0\oplus\g_\1$, then $\g_\1$ is a maximal Lie triple subsystem of $\g$.
\end{lemma}
\begin{proof}
Assume that $\g_\1\subsetneq T\subsetneq\g$ where $T$ is a Lie triple subsystem. 
We denote as before $T^2=[T,T]$.
Recall that $\g_\0=[\g_\1,\g_\1]$, because $[\g_\1,\g_\1]+\g_\1$ is a nonzero ideal of $\g$. 
Consequently $\g_\0\subset T^2$ and
$\g=\g_\0+\g_\1\subset T^2+T$. Hence
 we have $\g=T+T^2$ and from this, $T\cap T^2$ is an ideal of $\g$ because
 $$[\g, T\cap T^2]\subset[T,T\cap T^2]+[T^2,T\cap T^2]\subset T\cap T^2.$$
 Next $T\cap T^2=\g$ or $T\cap T^2=0$ but the first possibility yields the contradiction $T=\g$. Thus $T\cap T^2=0$ and we have 
 $\g=T^2\oplus T$. But this is also contradictory:
 $$\dim(\g)=\dim(\g_\0)+\dim(\g_\1)<\dim(T^2)+\dim(T)=\dim(\g).$$
 So $\g_\1$ is a maximal Lie triple subsystem.
 \end{proof}

 So far we have two sources of maximal Lie triple subsystems of a simple real or complex Lie algebra $\g$ (of finite dimension): (i) maximal subalgebras, (ii) odd parts of nontrivial $\Z_2$-gradings. We prove next that any maximal Lie triple subsystem of $\g$ is of any of these forms:

 \begin{proposition}\label{prop_todascasoalg}
 Let $\g$ be a
 simple real or complex Lie algebra $\g$ (of finite dimension) and $T$ a maximal Lie triple subsystem of $\g$. Then $T$ is either  a maximal subalgebra or the odd part of a $\Z_2$-grading on $\g$.
\end{proposition}
\begin{proof}
It $T^2\subset T$ then $T$ is a subalgebra and hence it is a maximal subalgebra (or otherwise $T$ would not be a maximal subtriple). So assume 
 $T^2\not\subset T$. Then $T\subsetneq T^2+T$ which is a subalgebra and in particular a subtriple. 
 By maximality of $T$, we have $\g=T^2+T$ so that $T^2\cap T$ is (as in the above lemma) an ideal of $\g$. Now the simplicity of $\g$ gives two possibilities:
 (i) $\g=T^2\cap T$ or (ii) $T^2\cap T=0$. Since $T\ne\g$ we have  $T^2\cap T=0$, so that $\g=T^2\oplus T$ and this is a nontrivial $\Z_2$-grading ($T^2\ne0$). 
\end{proof}

 Now we apply these results to our setting, the complex Lie algebra $\g_2$.

\begin{proposition}\label{pr_LTS_G2}
The maximal Lie triple subsystems of 
$\g_2=\der(\R^7,\Omega)=\{d\colon \R^7\to\R^7:\Omega(d(x),y,z)+\Omega(x,d(y),z)+\Omega(x,y,d(z))=0\}$, are just
 \begin{enumerate} 
  \item Any principal subalgebra;  
        \item $\h_2^\ell:=\{d\in\g_2:d(\ell)=0\}$, for a fixed $0\ne\ell\in\R^7$;  
        \item   $\h_4^V:=\{ d\in\g_2:d(V^\bot)\subset V^\bot\}$, for a fixed $V\in\M_S$;  
   \end{enumerate}
   the three of these subalgebras considered as LTS; and
   \begin{enumerate}    
        \item[(4)]   $\m_4^V:=\{ d\in\g_2:d(V^\bot)\subset V,\, d(V)\subset V^\bot \}$, for a fixed $V\in\M_S$.
    \end{enumerate}
    These LTS have dimensions 3, 8, 6 and 8, respectively.
    \end{proposition}

\begin{proof} 
According to the  renowned  paper \cite{Dynkin52} by E.~Dynkin in 1952, there are up to conjugation three maximal subalgebras   of the complex Lie algebra $\g_2^\C$: two of rank 2 (isomorphic to $\sl_3(\C)$ and $\mathfrak{so}_4(\C)$, see Tables 5 and 6 in \cite{Dynkin52}) and one principal  TDS (Table~16 in the same paper).
Of course, any principal TDS   jointly with   $\h_2^\ell$ and $\h_4^V$ are just representatives of the three conjugacy classes of  the  maximal subalgebras of $\g_2$ (of dimensions 3, 8 and 6 respectively). 
In fact, the maximality of any principal subalgebra has just been recalled in Section \ref{se_ppaldefs}. Besides, if $\ell=e_1$ (no loss of generality), the restriction map gives the isomorphism $\h_2^\ell\cong \mathfrak{su}(\ell^\bot, \sigma)\cong \mathfrak{su}(3)$ between simple Lie algebras, for $\C_\ell:=\R\oplus\R\ell$ real algebra isomorphic to the complex numbers,
$\ell^\bot=\C_\ell e_2\oplus\C_\ell e_3\oplus\C_\ell e_5$ a vector space over $\C_\ell$ of dimension $3$, and $\sigma\colon \ell^\bot\t\ell^\bot\to \C_\ell$ the Hermitian form given by $\sigma(x,y)= \esc{x,  y}1-\esc{\ell x,y}\ell$ (extracted from \cite[\S3.5]{procAlb}).
Also, if $V\in\M_S$, the restriction map gives the isomorphism  $\h_4^V\cong\mathfrak{so}(V^\bot,\esc{\cdot,\cdot })\cong\mathfrak{so}(4)$ between semisimple Lie algebras. 
The maximality of $\h_4^V$ can be concluded (alternatively to a direct comparison with \cite{Dynkin52}) from Lemma~\ref{le_maximalehr} and the fact that in the three cases (a  principal TDS, $\h_2^\ell$ and $\h_4^V$), 
there exists an $\h$-irreducible complement $\m$  (over $\R$, not absolutely irreducible in two of the cases)
such that $\g_2=\h\oplus\m$, taking into account the reductive decompositions analyzed in \cite[\S2]{LY}.  

On the other hand, for any $V\in\M_S$, the $\Z_2$-grading $\R^7=V\oplus V^\bot$ induces a $\Z_2$-grading on $\g_2=\der(\R^7,\t)$ with just $\h_4^V$ and $\m_4^V$ the even and the odd part, respectively. 
So, $\m_4^V$ is a maximal Lie triple subsystem by Lemma~\ref{lema2}.
Finally, the above LTS exhaust all the maximal Lie triple subsystems of $\g_2$ by Proposition~\ref{prop_todascasoalg}, because 
  there is only one order two automorphism of $\g_2$ up to conjugation, so that any $\Z_2$-grading on $\g_2$ is as above for a convenient associative subalgebra $V$ (the fixed part of the  automorphism producing the grading).
\end{proof}

  Note that the maximal subalgebras appeared above are more than just a technical tool, 
 but they are in fact crucial for the geometric parts following later.  Simply recall from Section~\ref{se_curvatura}
 that as (maximal) Lie triple systems they are the tangent spaces to the  (maximal) totally geodesic submanifolds of $G_2$.
 In fact, some of the Lie subalgebras shown here are easily recognizable. If $G$ is the Lie group of automorphisms of an algebra $A$, the different stabilizers $G_a$ (with $a\in A$) provide derivation algebras $\hbox{Lie}(G_a)$. This is nothing but the subalgebra of all derivations vanishing on $a$. 
 Also given a subspace $S$ of $A$, the Lie algebra of the subgroup of automorphisms preserving $S$ is just the derivations $d$ of $A$ such  that $d(S)\subseteq S$.

%%%%%%%%%%%%%%%%%%%%%%%%%%%%%%%%%%%%%%%%%%%%%%%%%%%%% 

\subsection{Maximal totally geodesic submanifolds of $G_2$}

As mentioned at the end of  Section~\ref{se_curvatura},
  the tangent space determines the totally geodesic submanifold. 
  Recall also that, if $\N$ is   totally geodesic, so are $L_g(\N)=\{gh:h\in\N\}$ and $R_g(\N)=\{hg:h\in\N\}$ for any $g\in G_2$.
  (Again $L_g,R_g\colon G_2\to G_2$ are the multiplication operators.)
  Besides, for any $\N$ there is some $g\in G_2$ (lots of such elements, in fact) such that  the neutral element in the group, $I=\id_{\R^7}$, belongs to $ L_g(\N)$. 
  Thus,
\begin{proposition}\label{pr_tgm_casoG2} 
    If $\N$ is a maximal totally geodesic submanifold of $G_2=\mathrm{Aut}(\R^7,\t)$, 
     then there is $g\in G_2$ such that $L_{g} (\N)$ equals
     either
    \begin{enumerate}
     \item $H^\h=\{F\in G_2:\hbox{Ad}(F)(\h)\subset\h\}$ for some principal subalgebra $\h$; or
        \item $H^\ell:=\{F\in G_2:F(\ell)=\ell\}$ for some $0\ne\ell\in\R^7$; or
        \item $H^V:=\{F\in G_2:F(V^\bot)\subset V^\bot\}$ for some $V\in\M_S$; or
       
        \item   $\widetilde\M_S:=\{\theta\in G_2\setminus\{I\}:\theta^2=\id_{\R^7}\}$.
    \end{enumerate}
    These manifolds (and their  translated manifolds by left and right multiplication operators) are respectively diffeomorphic to: $H^\h\cong \SO(3)$, 
    $H^\ell\cong \SU(3)$, $H^V\cong \SO(4)$ and $\widetilde\M_S\cong G_2/\SO(4)$.
\end{proposition}

\begin{proof}
All we have to do is to check that the related tangent spaces are those in Proposition~\ref{pr_LTS_G2}. 
Recall that $d\in T_IH$, for $H$ a subgroup of $G_2$, if $\exp(  td)\in H$ for all $t\in\R$.
First, $d\in T_IH^\h$ if $\hbox{Ad}(\exp(td))(\h)\subset\h$, that is, the vector subspace $\h$ is invariant for $\hbox{Ad}(\exp(td))=\exp(t\,\hbox{ad}d)$. We derive at $t=0$ to get $\hbox{ad}d(\h)\subset \h$. But any principal subalgebra is   selfnormalizing (proved, for instance, in \cite[proof of Proposition~3]{procAlb}), so that $[d,\h]\subset\h$ implies $d\in\h$.
Second,   $d\in T_IH^\ell$ if $\ell=\exp(td)(\ell)=\sum_{n=0}^\infty \frac{t^n}{n!}d^n(\ell)$. Taking $\frac{d}{dt}\vert_{t=0}$ we get $d(\ell)=0$. Also the third case is clear, since $\exp(td)$ leaving invariant the vector subspace $V^\bot$ for all $t$ implies that the same is true for $d$.

As regards the fourth set, there is a natural bijective correspondence   $\M_S\to\widetilde\M_S\subset G_2$, which assigns, to each $V\in\M_S$, the order 2 automorphism $\theta_V$ given by
\begin{equation}\label{eq_order2auto}
  \theta_V\vert_{V }:=\mathrm{id},\quad  \theta_V\vert_{V^\bot}:=-\mathrm{id}.
\end{equation} 
This is well defined, since recall that $\R^7=V \oplus V^\bot$ is a $\Z_2$-grading. Conversely, given $\theta\in\widetilde\M_S$, then $\hbox{Fix}(\theta)$ is an associative subalgebra. 
This identification is compatible with the $G_2$-action in $ \widetilde\M_S$ given by $(F,\theta)\mapsto \hbox{Ad}(F)(\theta)$.
To compute the tangent space to this manifold at a point $\theta_0\in\widetilde\M_S$, take a curve $\theta(t)\in \widetilde\M_S$ with $\theta(0)=\theta_0$ and compute $d=\theta'(0)$. 
Deriving $\theta(t)^2 =\mathrm{id}$ and evaluating at $0$ we get $d\theta_0+\theta_0d=0$. But this is equivalent to the fact  that $d$ interchanges $\hbox{Fix}(\theta_0)$
and $\hbox{Antifix}(\theta_0)$, that is, $d$ is an odd derivation. 

Finally, we have the group isomorphisms:
  $H^\ell\cong \SU(\ell^\bot,\s)\cong\SU(3)$ ($F\mapsto F\vert_{\ell^\bot}$) for $\s$ the Hermitian product considered in the proof of Proposition~\ref{pr_LTS_G2}; 
   $H^V\cong \SO(V^\bot,\esc{\cdot,\cdot })\cong\SO(4)$ ($F\mapsto F\vert_{V^\bot}$) with $\esc{\cdot,\cdot }$ as usual the scalar product in $\R^7$;
 and   $H^\h\cong \SO(\h,\kappa)$  ($F\mapsto \hbox{Ad}(F)\vert_{\h}$) for $\kappa$ the Killing form of $\g_2$.
The computations to check that the last map is really an isomorphism are developed in Remark~\ref{re_SOnoesfera}.
(It should be well known but there is some confusion on the existent literature and several authors speak about $\mathbb S^3$   instead of $\SO(3)$.)
\end{proof}

Note that  
$H^V$ coincides with $\{F\in G_2:F\theta_V=\theta_V F\}$, the centralizer of the order 2 automorphism $\theta_V$, and note too that every order two automorphism is $\theta_V$ for some $V\in\M_S$ (the eigenspace of the eigenvalue $1$ of the automorphism). It is also true that $H^V= \{F\in G_2:F\pi_V=\pi_V F\} $, for $\pi_V$ the orthogonal projection on $V$ as in Section~\ref{sub_M_S}.

\begin{remark}{\rm
This description of $\widetilde \M_S$ fits with the usual way of seeing   a symmetric space $G/H$ inside $G$. On one hand, take the geodesic symmetry $s_e\colon G\to G$ given by $s_e(g)=g^{-1}$ and consider its fixed points set $\{g\in G:g^2=\id\}$. The connected components are totally geodesic, but $\{e\}$ is an isolated connected component. 
On the other hand, if $H=\hbox{Fix}(\s)$ for $\s$ an order 2 automorphism of $G$, then we can take the (well-defined) embedding $G/H\to G$, $gH\mapsto \s(g)g^{-1}$. In our case, we have fixed  an order two automorphism $\theta$ of $(\R^7,\t)$, so that $\s=\hbox{Ad} (\theta)\in\aut (G_2)$ also has order two and $H=\hbox{Fix}( \s)=\{g\in G:\theta g=g\theta\}=H^V$ for $V=\hbox{Fix}(\theta)$. Thus the copy of $G_2/H^V$ inside $G_2$  is $\{\hbox{Ad} (\theta)(g)g^{-1}:g\in G_2\}=\{\theta g\theta g^{-1}:g\in G_2\}=L_{\theta}(\widetilde\M_S)$, since any order two automorphism of $G_2$ is conjugated to $\theta$.
}\end{remark}

\begin{remark}\label{re_polar}{\rm
Let us check how this fits with the classical results of  \cite{ChenyNagano}: 
If $M$ is a Riemannian symmetric space of compact type and $p\in M$, 
the connected components other than $\{p\}$ of $ \hbox{Fix}(s_p)$
 are called {\em polars} or $M_+$. 
 In our case, $\widetilde \M_S$ is  clearly a polar of $G_2$, and the other ones are $R_p(\widetilde \M_S)$ for $p\in G_2$. In fact, $s_p(q)=pq^{-1}p$ fixes $q$ just in case $(qp^{-1})^2=\id_{\R^7}$, that is, if $qp^{-1}\in\widetilde \M_S$ (if $p\ne q$), so that $q=R_p(qp^{-1})\in R_p(\widetilde \M_S)$.
Also,  for every polar $M_+$ of $M$ and every $q\in M_+$,
there exists another reflective submanifold $M_-$ of $M$ passing through $q$ such that  $T_qM_-=(T_qM_+)^\perp$,
 called a {\em meridian} of $M $.
 In our setting, this general fact can be easily checked: for $q$ in the polar $ M_+=  R_p(\widetilde \M_S)$, take as $V=\hbox{Fix}(qp^{-1})\in \M_S$ and $M_-=R_p(H^V)$. Indeed, 
  $q\in M_-$ since $qp^{-1}\in H^V$,  and moreover
  $T_qM_+$ and $T_qM_-$ are orthogonal and complementary, because they are, respectively, the even and the odd part of the $\Z_2$-grading on $\g_2$ induced by the automorphism $qp^{-1}$.
  (More general aspects on the so-called $(M_+,M_-)$-method   by Chen and Nagano can be consulted in  \cite[\S11.2.5]{Chen}.)
  }\end{remark}

\begin{remark}\label{re_ppales}
{\rm
While the $\ell$ and the $V$ in Proposition~\ref{pr_tgm_casoG2} are clear, 
principal subalgebras are difficult to figure out.
However, if we could somehow pick a TDS at random, the most likely TDS chosen would be a principal TDS (hence the name).
So paradoxically, the most common TDS is also the hardest to describe. The reason of this could be psychological: when one thinks of a subalgebra of a derivation algebra, one tends to consider derivations annihilating something, or derivations preserving a suitable subspace ﻿﻿(so we cannot consider our choices to be  random).
For instance, fixed $V\in\M_S$, one thinks of $\mathfrak s_1=\{d\in\g_2:d(V)=0\}$, which  is a TDS (see subsequent Eq.~\eqref{eq_Sl}), even  the more far-fetched 
$\mathfrak s_2=\{d\in\g_2:d(V^\bot)\subset V^\bot,[d,\mathfrak s_1]=0\}$ is a  TDS too (see  Eq.~\eqref{eq_Sr}), but none of these are principal. 
What happens is that   $\R^7$ is irreducible for any principal subalgebra, $\h$   leaves no  invariant subspace!
In the complex case, this   handicap is avoided using roots:   a principal subalgebra of an arbitrary complex simple Lie algebra 
can be easily constructed with root vectors, with  
the semisimple element $h\in\h=\esc{h,e,f}$  
chosen by doing $\alpha(h)=2$ for all $\alpha$ in a set of simple roots of a root decomposition relative to a Cartan subalgebra. 
But we cannot apply this to our compact case, and the description in terms of roots of $\h^\C$ does not seem to be of much help.
So, let us look for a explicit description of one principal subalgebra (recall that all of them are conjugated). A completely concrete election of a suitable $h$ can be consulted in \cite[Lemma~2.7]{procAlb}   and the related $e$ and $f$ in \cite[Proposition~2.8]{procAlb}, both in terms of the maps $D_{x,y}$. These maps are defined as derivations of octonions as $D_{x,y}(z)=[ [x,y],z]+3(x,z,y)$\footnote{The braces here are used for the \emph{associator}, that is, $(x,y,z)=(xy)z-x(yz)$. The fact that $\mathbb O$ is not associative means that the associator is not identically zero.}, what is not difficult to translate in terms of   $\g_2=\der(\R^7,\t)$ (recall Remark~\ref{re_octonions}). Thus, the above mentioned principal subalgebra is spanned by
\begin{equation}\label{eq_unappal} 
\begin{array}{l}
h_1=\frac16\big(4D_{\ell,\iu\t\ell}+5D_{\ju,\ku}\big),\\ 
h_2=\sqrt{\frac32}D_{\iu,\iu\t\ell}+\frac13\sqrt{\frac52}\big(D_{\ell,\ju}+D_{ \iu\t\ell,\ku}\big),\\ 
h_3=-\sqrt{\frac32}D_{\iu,\ell}+\frac13\sqrt{\frac52}\big(D_{\ell,\ku}-D_{\iu\t\ell,\ju}\big),
\end{array}
\end{equation}
for $\{\iu,\ju,\ku=\iu\times\ju\}$ an orthonormal basis of some $V\in\M_S$ and $\ell$ a unit vector in $V^\bot$.  
Be careful with the notations in Eq.~\eqref{eq_unappal}: i) We have interchanged the roles of $\ju$ and $\ell$ in \cite[Lemma~2.7]{procAlb}  to adapt this algebra to Section~\ref{se_2dimLTS_v2}; ii) we have changed the names of  $e$ and $f$  in the above reference to avoid mistakenly imagining that they are nilpotent elements; iii) we have multiplied those $e$ and $f$ with $\sqrt{\frac32}$ to get a nice typical basis of $\h\cong\mathfrak{so}(3)$. More precisely, for $i=1,2,3$ (sum modulo 3), it is tedious but easy to check that
the elements in \eqref{eq_unappal} satisfy
\begin{equation}\label{eq_formulitabasica} 
[h_i,h_{i+1}]=h_{i+2}.
\end{equation}
The three elements in the basis are of course semisimple (diagonalizable after extending scalars to $\C$). We can  even be more precise: $h_1\o 2\iu\in\g_2^\C=\g_2\o_\R\C$ acts on $\R^7\o\C$ (we do not write now $\C^7$ for avoiding extra confusion with $\iu\in\C$) with eigenvalues $\{0,\pm2,\pm4,\pm6\}$ and eigenvectors $\{\iu\o1,\ell\o1\pm(\iu\t\ell)\o\iu,\ju\o1\pm\ku \o\iu,(\ju\t\ell)\o1\pm(\ku\t\ell)\o\iu\},$ respectively. In particular it is clear that $h_1$ acts absolutely irreducibly on $\R^7$ (moreover, the three $h_i$'s do). Note   that this implies that $\hbox{ad} h_1\colon\g_2\to\g_2$ is semisimple too, with $\h$ and $\h^\perp$ the only $\h$-invariant subspaces of $\g_2$, and moreover $\h^\perp$ is absolutely $\h$-irreducible ($h_1\o 2\iu$ acting with integer eigenvalues $\{10-2n:n=0,\dots,10\}$).
  }\end{remark}

\begin{remark}\label{re_SOnoesfera}
{\rm 
Note that there remains some confusion with the manifold $H^{\mathfrak h}$ in Proposition~\ref{pr_tgm_casoG2}. In some references, it is said to be $\mathbb S^3$, as in  \cite{Klein09} (see Table~\ref{tab:Klein}). This seems 
  to be a minor slip  because the same author in the more complete reference \cite{Klein} mentions  $\R P^3$ as one of the maximal totally geodesic submanifolds of $G_2/\SO(4)$.   
Thus we believe it is important to complete the  details in the proof of Proposition~\ref{pr_tgm_casoG2}
to elucidate once and for all which of the two possibilities it is. 
That is, 
we want to prove that the map $H^\h\to\SO(\h,\kappa)$ given by $F\mapsto\hbox{Ad}(F)\vert_\h$ is indeed a group isomorphism. To do so, it is sufficient to prove 
  that, if $F\in G_2=\aut(\R^7,\t)$ commutes with $\{h_i:i=1,2,3\}$ in Eq.~\eqref{eq_unappal}, then necessarily $F=I $ (the identity of $\R^7$). The fact that $F$ commutes with $h_1$ implies that $F$ preserves $\ker(h_1)=\esc{\iu}$, $\ker(h_1-I)^2=\esc{\ell,\iu\t\ell}$,
  $\ker(h_1-2I)^2= \esc{\ju,\ku}$ and $\ker(h_1-3I)^2=\esc{\ju\t\ell,\ku\t\ell }$. As $F$ is an isometry, then $F(\iu)=\pm\iu$. Note that $F$ is determined by $F(\iu)$, $F(\ju)=a\ju+b\ku$ with $a^2+b^2=1$, and $F(\ell)=(c+d\iu)\t\ell$ with $c^2+d^2=1$. 
  There is no possibility for $(a,b,c,d)$  with $F(\iu)=-\iu$: in this case $F(\ku)=b\ju-a\ku$, and similarly we compute $F(\iu\t\ell)$, $F(\ju\t\ell)$ and $F(\ku\t\ell)$. Now $F$ commuting with $h_2$ leads to $(a,b,c,d)=(1,0,-1,0)$, while $F$ commuting with $h_3$ leads to $(a,b,c,d)=(1,0,1,0)$, and both things cannot occur simultaneously.  
  On the other hand, proceeding in the same way to   find $(a,b,c,d)$ in case $F(\iu)=\iu$, we can conclude that   the only solution is that one with $(a,b,c,d)=(1,0,1,0)$, corresponding to the identity map.
}\end{remark}

%%%%%%%%%%%%%%%%%%%%%%%%%%%%%%%%%%%%%%%%%%%%%%%%%%%%%%%%

\section{Maximal totally geodesic submanifolds of $G_2/\SO(4)$}  \label{se_stg_MS}

We will study the 8-dimensional symmetric space $G_2/\SO(4)$ from two different viewpoints. 
One of them, in the subsequent Section~\ref{se_stg_MS_free}, or the octonionic-free version, uses the description in Eq.~\eqref{emese} as a submanifold of a Grassmannian, and in this way the description of the totally geodesic submanifolds becomes very natural. 
Furthermore, anyone without knowledge on cross products or octonions can follow the arguments, because the related Lie triple systems
 are completely described in matricial terms (including the triple products).
The disadvantage  is that it is  difficult to prove directly in this context that the found   maximal LTS cover all the possibilities.
The other point of view (the one that we will develop in this section)   uses the description of $G_2$ as the automorphism group of  $(\R^7,\t)$ (a disadvantage or not depending on one's background), and this permits us to apply some general lemmas about maximality to achieve the desired classification in Theorem~\ref{pr_LTS_MSv1}. The main tool is the fact that if $T$ is the Lie triple system related to $\M_S$, 
then its standard envelope $\g(T)$ is the Lie algebra of type $\g_2$. 
Roughly speaking, we can extract more information by thinking of the symmetric space inside $G_2$
rather than inside the Grassmannian $\grass\nolimits_3(\R^7)$ because   in the latter case there is more \emph{clearance}.

%%%%%%%%%%%%%%%%%%%%%%%%%%%%%%%%%%%%%%%%%%%%%%%%%%%%%%%%%%

\subsection{Lie triple system related to the symmetric space $\M_S$  and its envelope} 

Recall that $G_2$ acts on $\M_S=\{V\le\R^7:\dim V=3,\,V\t V\subset V\}$ and the isotropy group of $V$
 is $H^V\cong\SO(V^\bot,\esc{\cdot,\cdot })$ defined in Proposition~\ref{pr_tgm_casoG2}. Thus $\M_S\cong G_2/\SO(4)$. If $\theta_V$ is the order two automorphism of $\R^7$ given by
\eqref{eq_order2auto}, then $\hbox{Ad}(\theta_V)\colon \g_2=\der(\R^7,\t)\to\g_2$ is an order two automorphism and $\hbox{Fix}(\hbox{Ad}(\theta_V))=\{d\in\g_2:d\theta_V=\theta_Vd\}=\h_4^V\cong \mathfrak{so}(V^\bot,\esc{\cdot,\cdot })$, with the notations  in Proposition~\ref{pr_LTS_G2}, is the even part of the induced $\mathbb{Z}_2$-grading. The odd part of the grading   can be identified with the tangent space at the point $V$,
$$
T_V\M_S=\{d\in\g_2:d(V)\subset V^\bot, d(V^\bot)\subset V\}=\m_4^V.
$$ 
Thus $\g_2$ is just the standard envelope algebra of the Lie triple system $T_V\M_S$.

\subsection{Adapted principal subalgebras}\label{se_adapted}
The following concept plays a key role  to describe the 2-dimensional maximal Lie triple subsystems,  the tangent spaces to the maximal totally geodesic submanifolds that are \lq placed in a skew position\rq\, according to \cite{Klein}.

\begin{definition}
 If $\h$ is a principal subalgebra of $\g_2$,    and $V\in \M_S$, we say that $\h$ is \emph{adapted to $V$} if $\h$ is homogeneous relative to the $\Z_2$-grading on $\g_2$ produced  by $V$, that is, any $d\in\h$ can be written as $d=d_\0+d_\1$, with $d_i\in\h $  
 and $ d_i(V_j) \subset V_{i+j}$ for any $i,j=\0,\1$, sum modulo 2. (We denote $V_\0=V$ and $V_\1=V^\bot$.) 
\end{definition}

Fixed a principal subalgebra $\h$, we can always find $V\in\M_S$   to which $\h$ is adapted, and similarly,
fixed $V\in\M_S$, we can always find a principal subalgebra adapted to $V$.  
Indeed, as all principal subalgebras are conjugated by $G_2$, and the same happens to any associative subalgebra of $(\R^7,\t)$,
then, 
in order to check the assertion it is enough to find one principal subalgebra adapted to some $V$. 
We can observe  that   
 the principal subalgebra considered  in Eq.~\eqref{eq_unappal}  (extracted with some minor changes from \cite[Lemma 2.7,\,Proposition 2.8]{procAlb})
 is precisely adapted to   $V=\esc{\{\iu,\ju,\ku\}}$, with $h_1$ an even derivation, and $h_2$ and $h_3$ two odd derivations.
Notice the following   characterization, important for us when is applied to principal subalgebras.

\begin{lemma}\label{le_laparteimpardim2}
If $\h$ is any principal subalgebra, $\h$ is  adapted to $V\in \M_S$ if and only if $\dim(\h\cap\m_4^V)=2$.
\end{lemma}

\begin{proof}
First, assume $\dim(\h\cap\m_4^V)=2$ and let us prove that $\h=(\h\cap\h_4^V)\oplus (\h\cap\m_4^V)$. Otherwise, by dimension count,  
$\h\cap\h_4^V=0$ so that $\h\cap\m_4^V$ would be an abelian two-dimensional subalgebra of $\h$. 
This is a contradiction since the maximal abelian subalgebras of a TDS have dimension 1 (the rank of $\h$). Conversely, assume that $\h=(\h\cap\h_4^V)\oplus (\h\cap\m_4^V)$. Denote by $s=\dim(\h\cap\m_4^V)$. 
If $s=0$, then $\h\subset \h_4^V$ and $\h$ would preserve $V$, a contradiction since $\R^7$ is an irreducible $\h$-module (recall Remark~\ref{re_ppales}).
If $s=3$, then $\h\subset\m_4^V$, so that $\h=[\h ,\h ]\subset \h\cap\h_4^V=0$, again a contradiction. 
Finally, if $s=1$, then $[\h\cap\m_4^V,\h\cap\m_4^V]=0$ and $[\h\cap\m_4^V,\h\cap\h_4^V]\subset \h\cap\m_4^V$ due to the grading, so that
$\h\cap\m_4^V$ would be a one-dimensional ideal of the simple algebra $\h$, a contradiction. In conclusion, the only possibility is $s=2$.
\end{proof}

To characterize when an associative subalgebra is adapted to a fixed principal TDS, the next result is in order. 

\begin{lemma}\label{le_utiladaptada}
    Given $V\in\M_S$ and a principal subalgebra $\h$  of $\g_2=\der(\R^7,\t)$, then $\h$ is adapted to $V$ if and only if $\theta_V\in H^\h$,
    for $\theta_V\in G_2$ the order two automorphism considered in Equation~\eqref{eq_order2auto}.
\end{lemma}
\begin{proof}
Let us denote by $p_V\colon \g_2=\h_4^V\oplus\m_4^V\to\m_4^V$ the projection  on the odd part of the $\Z_2$-grading.
In these terms,
$\h$ is adapted to $V$ if and only if $p_V(\h)\subset \h$, because if the projection on the odd part continues in $\h$, this is also true for the projection on the even part. Now observe the following fact: if $d\in\g_2$, 
and we write $\theta=\theta_V\in G_2$,  then $d$ is even (respectively, odd) if and only if 
$\theta d \theta=d$ (respectively, $\theta d \theta=-d$). Hence the projection $p_V$ is given by $p_V(d)=\frac{d-\theta d\theta}{2}$. Thus, if $d\in\h$, the condition $p_V(d)\in\h$ is equivalent to $\mathrm{Ad}( \theta)(d)=\theta d \theta\in\h$ ($\theta$ has order two), that is, equivalent to $ \theta\in H^\h$.
\end{proof}

This allows us to relate the different associative subalgebras adapted to a fixed principal subalgebra.

\begin{proposition}\label{pr_lasadaptadas}
   Let $\h$ be a principal subalgebra adapted to $V_0\in\M_S$. For any $V\in\M_S$, $\h$ is adapted to $V$ if and only if there is $F\in H^\h$ with $F(V_0)=V$. 
    \end{proposition}
    
\begin{proof}
First, assume that $V=F(V_0)$ for $F\in H^\h$. Simply observe that $\theta_{F(V_0)}=F\theta_{V_0}F^{-1}$. As Lemma~\ref{le_utiladaptada} says that 
$\theta_{V_0}\in H^\h$, but $F$ and $F^{-1}$ also belong to the group $H^\h$, thus $\theta_{F(V_0)}\in H^\h$ and $\h$ is adapted to $F(V_0)$.
Conversely, take $V\in\M_S$ such that $\h$ is adapted to $V$.
    Again by Lemma~\ref{le_utiladaptada}, $\theta_{V}$ and  $\theta_{V_0}$ both belong to $H^\h$. 
    But the group $H^\h $ is isomorphic to $\SO(3)$ (see Remark~\ref{re_SOnoesfera}), and in $\SO(3)$ all order two elements are  conjugated.  
    Indeed, any element in a compact connected Lie group is conjugated to an element in its maximal torus, and the maximal torus of $\SO(3)$ is $\SO(2)$ (\cite{Adams}). Hence there is $F\in H^\h$ such that 
    $F \theta_{V_0}F^{-1}=\theta_{V}$. Thus $\theta_{V}=\theta_{F(V_0)}$, which implies $V=F(V_0)$, 
    because   we can recover $V\le\R^7$ from the automorphism as $\hbox{Fix}(\theta_V)$.\end{proof}

 Soon (in Corollary~\ref{co_LTS_MSv1}) we will prove that, for  a principal TDS  $\h$ adapted to $V\in \M_S$, then $\h\cap\m_4^V$ is a maximal subtriple of $ \m_4^V$.

%%%%%%%%%%%%%%%%%%%%%%%%%%%%%%%%%%%%%%%%%%%%%%%%%%%%%%
\subsection{Maximal Lie triple subsystems}

 Another general result about maximality of Lie triple systems  follows.

\begin{proposition}\label{prop_todascasosegundo}

 Let $\g=\g_{\bar0}\oplus\g_{\bar1}$ be a $\Z_2$-grading on a simple Lie algebra $\g$, and $\h$  a maximal
Lie subalgebra of $\g$ such 
that   $\h=(\h\cap \g_{\bar0})\oplus(\h\cap\g_{\bar1})$.
If $[\h\cap\g_{\bar1},\h\cap\g_{\bar1}]=\h\cap\g_{\bar0}$,   
then $\h\cap\g_{\bar1}$ is a maximal Lie subtriple of $\g_{\bar1}$.  

The required assumption on the bracket is automatically fulfilled  if $\h$ is simple and different from $\g_{\bar0}$.
  \end{proposition}

 \begin{proof} Let $T=\g_{\bar1}$ and $\h$ a graded (or homogeneous) subalgebra, that is, compatible with the grading on $\g$. The condition $[\h\cap\g_{\bar1},\h\cap\g_{\bar1}]=\h\cap\g_{\bar0}$ means that $\g(\h\cap T)$,
    the   envelope of the triple $\h\cap T$, fills the whole $\h$. Assume now there is  another Lie triple system  $T'$ with $\h\cap T\subsetneq T'\subset T$.
     Then we have a chain of standard envelopes $\h=\g(\h\cap T)\subset  \g(T')\subset \g(T)=\g$. The maximality of $\h$ as a subalgebra gives that 
     either $\g(T')=\h$ or $\g(T')= \g$. The first case is not possible, due to the grading $\g(T')=T'\oplus (T')^2$. Thus the situation is the second one, and $T'=T$.
     
The simplicity of $\h$ is a sufficient condition to guarantee the  assertion $\g(\h\cap T)=\h$ when $ \h\cap T\ne 0$,    because  $\g(\h\cap T)$ is an ideal of $\h$.
If $ \h\cap T=0$, then $\h=\g_{\bar0}$ and  we of course exclude this situation.
\end{proof} 

Now we can apply this result to our setting, for $\g=\g_{\bar0}\oplus\g_{\bar1}$ the $\Z_2$-grading on $\g_2$ induced by $V\in\M_S$, i.e. 
  $\g_{\bar0}=\h_4^V$ and $\g_{\bar1}=T_V\M_S$, and the maximal subalgebras of $\g_2$ described in 
Proposition~\ref{pr_LTS_G2}. Perhaps the more subtle case is when $\h$ is a principal subalgebra, 
that is why we have devoted Section~\ref{se_adapted} to study adapted principal TDS. 
Thus, if a principal subalgebra $\h$ is adapted to $V$, then $\h\cap  \m_4^V$ is a 2-dimensional maximal subtriple of $   \m_4^V$ (Lemma~\ref{le_laparteimpardim2}
and Proposition~\ref{prop_todascasosegundo}).  
We can similarly apply the above proposition  to  the remaining maximal Lie subalgebras of $\g_2$.

\begin{corollary}\label{co_LTS_MSv1}   
For any $V\in\M_S$,  a principal subalgebra $\h_1$ adapted to $V$, $0\ne\ell\in V^\bot$, $0\ne\iu\in V $, and $W\in\M_S$ such that both $W\cap V $ and   $W\cap V^\bot $ are nonzero, then 
$$
\h_1\cap  \m_4^V,\quad \h_2^\ell\cap  \m_4^V,\quad \h_2^{\iu}\cap  \m_4^V ,\quad \h_4^W \cap  \m_4^V  ,\quad
$$
are maximal Lie triple subsystems of $\m_4^V$ ($\cong  T_V\M_S$) of dimensions $2$, $5$, $4$ and $4$ respectively.
 \end{corollary}

 \begin{proof}
In order to apply Proposition~\ref{prop_todascasosegundo}, we only have to check that 
$\h =\big(\h \cap \h_4^V )\,\oplus\,\big(\h \cap  \m_4^V)$ for $\h\in\{\h_2^\ell,\h^{\iu}_2,\h_4^W\}$,  as well as the fact that the envelope  $\g(\h_4^W \cap  \m_4^V)=\h_4^W$.
The first assertion reduces to a computation of dimensions. The paper \cite{LY} makes most of these computations.
In order to give more details, we will again use the operators $D_{x,y}\vert_{\R^7}$ as in Remark~\ref{re_ppales}, 
we fix a unit vector $\ell\in V^\bot$ 
and we consider the derivations $\lambda_a\colon \R^7\to \R^7$ ($a\in V$)
defined by
\begin{equation}\label{eq_Sl}  
\lambda_a(v)=0,\quad
\lambda_a(\ell)=a\t \ell,\quad  
\lambda_a(v\t \ell)=(a\times v)\times \ell-\esc{v,a}\ell,     
\end{equation}
for any $v\in V$.
It turns out that $\{d\in\der(\R^7,\t):d(V)=0\}=\{\lambda_a\mid a\in V\}\cong (V,\t)$ is  a three dimensional simple subalgebra. Besides 
$\h_4^V=\esc{\{ \lambda_a,\rho_a:a\in V  \}}$, for the derivations $\rho_a\colon \R^7\to \R^7$ ($a\in V$) given by
\begin{equation}\label{eq_Sr}  
\rho_a(v)= 2a\t v,\quad
\rho_a(\ell)=-a\t \ell,\quad 
  \rho_a(v\t \ell)= (a\times v)\times \ell+\esc{v,a}\ell.
\end{equation}
It is a straightforward computation that $[\lambda_a,\lambda_b]=2\lambda_{a\t b}$, $[\lambda_a,\rho_b]=0$, and  
$[\rho_a,\rho_b]=2\rho_{a\t b}$  for all $a,b\in V$. 
(To be precise, $\lambda_a$ and $\rho_a$ are, respectively,  the restrictions to $\R^7$ of the derivations of octonions $d_a$ and $D_a$ in \cite[Eq.~(2.9)]{LY}.)
This shows the two simple ideals of $\h_4^V$. Now  the required dimensions are   easy to compute. 
For $V$, $W$, $\iu$ and $\ell$ as in the statement of the corollary,

\begin{itemize} 
\item We know that $\dim\h_2^\ell=8$ (isomorphic to $\mathfrak{su}(3)$). Now we find
$$
 \h_2^\ell\cap (\g_2)_{\bar0}=\{\lambda_a+\rho_a:a\in V\}, 
 $$
a 3-dimensional subalgebra (naturally isomorphic to  $(V,\t)$), and
$$
  \h_2^\ell\cap  \m_4^V =\{d\in(\g_2)_{\bar1}:d(\ell)=0\} =\esc{\{D_{v,v\t\ell}:v\in V\}} , $$
of dimension 5 (isomorphic as a vector space to the space of symmetric tensors $S^2(V)$). 
\item Again $\dim\h_2^{\iu}=8$ and we are computing $\dim(\h_2^{\iu}\cap (\g_2)_{\bar i})$ for $i=0,1$.
First   check that 
$$
\h_2^{\iu}\cap (\g_2)_{\bar0}=\{\lambda_a:a\in V\}\oplus\esc{\{\rho_{\iu}\}},
$$
clearly of dimension 4. Then take a unit vector $\ju\in V\cap\iu^\bot$ to  compute explicitly
$$
\h_2^{\iu}\cap (\g_2)_{\bar1}= \{D_{\iu\t\ju,v\t\ell}-D_{\ju,(v\t\iu)\t\ell}:v\in V\} \oplus\esc{
\{  D_{\iu\t\ju, \ell}-D_{\ju, \iu\t\ell} \} } ,
$$
which is     4-dimensional too.
\item If we choose nonzero vectors $\ell$ and $\iu$ such that $\ell\in W\cap V^\bot$ and $\iu\in W\cap V $, then  
\begin{equation}\label{eq_parademo1}
W\cap V^\bot=\esc{\{ \ell,\iu\t\ell  \} },\quad
W\cap V=\esc{\{\iu\}},
\end{equation}
and, for any $0\ne\ju\in V$, $\ju\perp\iu$, we have
\begin{equation}\label{eq_parademo2}
W^\bot\cap V=\esc{ \{\ju,\iu\t\ju  \}},\quad
W^\bot\cap V^\bot=\esc{\{\ju\t\ell,(\iu\t\ju )\t\ell\}},
\end{equation}
since $\ell$ orthogonal to $V=\esc{\{\iu,\ju,\iu\t\ju  \}}$ implies that both $\ju$ and  $\iu\t\ju $ are orthogonal to $\ell$ and $\iu\t\ell$.
Now $\h_4^W=D_{W,W}+D_{W^\bot,W^\bot}\cong\mathfrak{so}(4)$ trivially decomposes as the   sum of 
\begin{equation}\label{eq_trocitoauxiliar}
     \h_4^W \cap  \m_4^V=D_{V\cap W,V^\bot\cap W}+D_{V\cap W^\bot,V^\bot\cap W^\bot} 
\end{equation}
and the algebra $\h_4^W \cap \h_4^V$, which is the sum of
$D_{V\cap W,V\cap W }$ (zero) and the one-dimensional subspaces $D_{V^\bot\cap W, V^\bot\cap W}$, $D_{V^\bot\cap W^\bot,V^\bot\cap W^\bot}$and $D_{V \cap W^\bot,V \cap W^\bot}$. Here we have used   that $D_{x,x}=0$ while $D_{x,y}\ne 0$ if $x$ and $y$ are linearly independent. Thus 
$$
\h_4^W \cap \h_4^V=\esc{\{D_{\ell ,\iu\t \ell},D_{\ju\t\ell ,\iu\t(\ju\t\ell) },D_{\ju ,\iu\t \ju}\}},
$$
 which has dimension 2 since, for any $x,y,z\in\R^7,$ 
 \begin{equation*}\label{eq_sumaciclica}
 D_{x\t y,z}+D_{y\t z,x}+D_{z\times x,y}=0.
\end{equation*}
 (In fact, $ \h_4^W \cap \h_4^V$ is abelian, a Cartan subalgebra of $\h_4^W$.) 
Coming back to the other piece of $\h_4^W$ \eqref{eq_trocitoauxiliar},  the above cyclic identity   jointly with the fact that 
  $V^\bot\cap W=(V \cap W^\bot)\t(V^\bot\cap W^\bot)    $ and both 
  $V \cap W^\bot$ and $V^\bot\cap W^\bot$ remain invariant under multiplication with $V\cap W$, permit us to see that the first summand in \eqref{eq_trocitoauxiliar} is contained in the second one, and hence   
\begin{equation}
    \label{eq_trocitoauxiliarv2}
  \h_4^W \cap  \m_4^V= D_{V\cap W^\bot,V^\bot\cap W^\bot}.
\end{equation}
  Thus $\h_4^W \cap  \m_4^V$   is naturally isomorphic (as a vector space)
  to $(V\cap W^\bot)\o(V^\bot\cap W^\bot)$ ($x\o y\mapsto D_{x,y}$) and so has dimension 4. 
\end{itemize}
Finally we need to check that $[\h_4^W \cap  \m_4^V,\h_4^W \cap  \m_4^V]=\h_4^W \cap  \h_4^V$. 
The   quickest way to check this (alternative to the use of \eqref{eq_trocitoauxiliarv2})
is to note that $\h_4^W \cap  \h_4^V=\esc{\{ \lambda_{\iu},\rho_{\iu}\}}$
while $\m_4^W \cap  \h_4^V=\esc{\{ \lambda_{\ju},\rho_{\ju},\lambda_{\iu\t\ju},\rho_{\iu\t\ju}\}}$, so that
$[\h_4^V \cap  \m_4^W,\h_4^V \cap  \m_4^W]=\h_4^V \cap  \h_4^W$ and then interchange the roles of $V$ and $W$.
\end{proof}

\begin{remark}
  {\rm   
  For any $\Z_2$-grading on an arbitrary  Lie algebra $\g=\g_{\bar0}\oplus\g_{\bar1}$, the even and the odd components are orthogonal with respect to the Killing form.
  In particular the Lie triple systems $\h_4^W \cap  \m_4^V$ and $\m_4^W \cap  \m_4^V$ are orthogonal, so that they correspond to  reflective totally geodesic submanifolds. 
Recall that a submanifold is reflective if and only if it is a connected component
of the fixed point set of an involutive isometry \cite{Leung}.  
In symmetric spaces, being a reflective totally geodesic submanifold is equivalent to the fact that the orthogonal of the related triple system is another LTS \cite[Proposition 11.1.5]{gallegosvarios}. Note that this agrees with Remark \ref{re_polar}.  
}
\end{remark}  
 
This corollary proves that all the found LTS are  maximal.  The fact that we have described \emph{all} the maximal Lie triple subsystems up to conjugation  could have been concluded 
by comparison of dimensions with the ones
from the works   
\cite{Klein} and \cite{Kollross}, which provided a list of maximal LTS in terms of roots and Satake diagrams.  But we can state a stronger result: they are not only all the maximal LTS of $\m_4^V$ up to conjugation, but they are \emph{all} the maximal LTS  of $\m_4^V$.   An elementary    proof is achieved independently of the cited papers, besides the fact that computations in \cite{Klein} were developed with Maple.

\begin{theorem}\label{pr_LTS_MSv1}   
 Fixed $V\in\M_S$, the maximal Lie triple subsystems of $\m_4^V$ ($\cong  T_V\M_S$) are exactly:
 \begin{itemize} 
  \item[(1)] $\h_1\cap  \m_4^V$, for some principal TDS $\h_1$ adapted to $V$; 
  \item[(2)] $\h_2^\ell\cap  \m_4^V$, for a fixed $0\ne\ell\in V^\bot$;
  \item[(3)] 
             $\h_2^{\iu}\cap  \m_4^V$, for a fixed $0\ne\iu\in V $; \vspace{2pt}
        \item[(4)]   $\h_4^W \cap  \m_4^V$, for a fixed $W\in\M_S$ such that both $W\cap V $ and   $W\cap V^\bot $ are nonzero. 
           \end{itemize}
  \end{theorem}

 \begin{proof}
Let $T$ be a maximal Lie triple system of $\m_4^V$. 
 As $T\subset \m_4^V$, then $T\subset \g(T)=T\oplus T^2\subset \g(\m_4^V)=\g_2$, and we can take 
  a subalgebra $\h$ of $\g_2$ maximal with the property   that $T\subset\h$. 
  Note that $\h$ is proper subalgebra of $\g_2$. This implies that $\h$ is maximal subalgebra of $\g_2$ and Proposition \ref{Leo} gives that it is a maximal Lie triple subsystem of $\g_2$. 
  Applying Proposition~\ref{pr_LTS_G2}, either $\h$ is principal, or there is $0\ne u\in\R^7$ with $\h=\h_2^u$ or there is $W\in\M_S$ with $\h=\h_4^W$.
  In any of the cases we have $T\subset \h\cap\m_4^V$, so that the maximality of $T$ among the LTS contained in $\m_4^V$ gives  $T= \h\cap\m_4^V$.
  We would like to prove that we can choose, respectively,  the principal subalgebra adapted to $V$, the   vector $u\in V\cup V^\bot$, and the associative subalgebra $W$ intersecting both $V$ and $V^\bot$, since otherwise maximality would fail. 
  Let's prove the veracity of this assertion according to the type of subalgebra $\h$.
  Denote, as usual, the $\Z_2$-grading with odd part $\m_4^V$ by $\g_2=\g_{\bar0}\oplus\g_{\bar1}$, 
  and recall that the standard envelope of $\m_4^V$ is the whole $\g_2$.
\smallskip
  
  \boxed{\textrm{$\h $ principal} } 
  If $T$ is contained in a principal TDS $\h_1$, note that $\dim T\le 2$. 
  Otherwise, $T=\h_1\subset\m_4^V$ and $\h_1=[\h_1,\h_1]\subset \g_{\bar0}\cap\g_{\bar1}=0$, an absurd. 
  If $\dim T=2$, then $\h_1$ is adapted to $V$ 
  by
  Lemma~\ref{le_laparteimpardim2} and we have finished.  
  Let us prove  by reductio ad absurdum that the situation 
  $\dim T= 1$, with $T$   maximal LTS in $\m_4^V$, is not possible. As in the proof of Lemma~\ref{le_LTSapatadas}, we can take $T=\esc{\{h_1\}}$ for some nice basis $\{h_1,h_2,h_3\}$   of $\h_1$ such that  $[h_i,h_{i+1}]=h_{i+2}$ for all $i$ (sum modulo 3). 
   Write $h_2=a_0+a_1$ and $h_3=b_0+b_1$ with $a_i,b_i\in(\g_2)_{\bar i}$.   Since $\h_1\cap\m_4^V$ is a LTS containing $T$, then $\h_1\cap\m_4^V=T$, so that 
  both the projections of $h_2$ and $h_3$ on the even part, i.e. $a_0$ and $b_0$, are nonzero ($h_2$ and $h_3$ belong to $\h_1\setminus T$).
  Projecting the relations $[h_i,h_{i+1}]=h_{i+2}$ on $\g_{\bar0}$ and  $\g_{\bar1}$, we get 
  $$
  [h_1,a_0]=b_1,\ 
  [h_1,a_1]=b_0,\ 
  [h_1,b_0]=-a_1,\ 
  [h_1,b_1]=-a_0. 
  $$
  In particular,   $a_1$ and $b_1$ are nonzero too. Note that $[h_1,[b_0,a_1]]=[-a_1,a_1]+[b_0,b_0]=0$, so that 
  $\esc{\{h_1,[b_0,a_1]\}}\subset \m_4^V$ is abelian, in particular it is both a subalgebra and a LTS.  As it contains $T$, the maximality of $T$ implies $[b_0,a_1]\in T$. But then it is easy to check that $\esc{\{h_1,a_1 \}}\subset \m_4^V$ is a LTS too: simply because $[h_1,a_1,h_1]=a_1$ and $[h_1,a_1,a_1 ]=[b_0,a_1]\in  T=\R h_1$. Again the maximality of $T$ as triple subsystem of $\m_4^V$ gives $a_1\in T=\R h_1$. Hence $b_0=[h_1,a_1]=0$, a contradiction.
  \smallskip

  \boxed{\textrm{$\h=\h_2^u $} } If $u\in V\cup V^\bot$, we are done. Otherwise, $u=u_1+u_2$ with $u_1\in V$, $u_2\in V^\bot$, both $u_i\ne0$.
  For any $d\in T$, we know $0=d(u)=d(u_1)+d(u_2)$, but $d(u_1)\in V^\bot$ and $d(u_2)\in V$, so that $d(u_1)=d(u_2)=0$. If we take $W=\esc{\{u_1,u_2,\u_1\t u_2\}}\in\M_S$, then $d(W)=0$ for all $d\in T$. Thus $T$ is contained in the LTS  $\h_4^W\cap \m_4^V$, and by maximality $ T= \h_4^W\cap \m_4^V$. But the
  associative subalgebra $W$ we have chosen satisfies that $W$ intersects both $V$ and $V^\bot$, so that $T$ would have dimension $4$ as in  the proof of Corollary~\ref{co_LTS_MSv1}. This is a contradiction since $T\subset \{d\in\g_2:d(W)=0\}$, which is a 3-dimensional simple subalgebra, as mentioned also in Corollary~\ref{co_LTS_MSv1}  after Eq.~\eqref{eq_Sl}.
   \smallskip
  
  \boxed{\textrm{$\h=\h_4^W $} } 
  If both $W\cap V $ and   $W\cap V^\bot $ are nonzero, we are done. If $W=V$, $ T= \h_4^W\cap \m_4^V=0$ is not maximal.
   Otherwise, we begin by  proving  that we have one of the next situations:
  \begin{itemize}
  \item[ a)] $V\cap W\ne 0$ and $W\cap V^\bot =0$;
  \item[ b)] $0=V\cap W $ and $\dim(V^\bot\cap W)=\dim(V\cap W^\bot)=\dim(V^\bot\cap W^\bot)=1$;
  \item[ c)] $0=V\cap W= V^\bot\cap W=V\cap W^\bot $ and $\dim(V^\bot\cap W^\bot)=1$.
  \end{itemize}
  In fact, it is not possible  $V^\bot\cap W^\bot=0$, since then   the dimension of the subspace $V^\bot\oplus W^\bot $ would be $8$ (in $\R^7!$). 
  If $0=V\cap W $, necessarily $\dim(V^\bot\cap W^\bot)=1$ because if $\dim(V^\bot\cap W^\bot)\ge2$, then
  $ 0\ne (V^\bot\cap W^\bot)^{\t2}\subset V\cap W$, getting a contradiction.
  (Take into account that the cross product of two independent elements is always nonzero.) Moreover, $(V^\bot\cap W )^{\t2}\subset V\cap W$ and $(V \cap W^\bot)^{\t2}\subset V\cap W$, so that both 
  $V^\bot\cap W$ and $V\cap W^\bot$ have at most   dimension 1. It remains to prove that $V^\bot\cap W$ and $V\cap W^\bot$ have the same dimension ($0$ or $1$). This is easy to achieve: if, for instance, $V^\bot\cap W\ne0$, then $0\ne(V^\bot\cap W)\t(V^\bot\cap W^\bot)\subset V\cap W^\bot$
  (and similarly if we begin with $V\cap W^\bot\ne0$). We next discuss how to find a suitable subalgebra according to the three different possible cases.

First, if the situation is $a)$,   the dimension of $V\cap W$ is not $3$ ($V\ne W$), and it is not $2$ because any two independent vectors in an associative subalgebra $V$ generate $V$ as an algebra. So it is one-dimensional and we can find a nonzero  vector with $V\cap W=\R \iu$.  
Now, for any $d\in T=\h_4^{W}\cap\m_4^V$,  we have $d(W)\subset W $ and $d(V)\subset V^\bot$, so that $d(\iu)\in W\cap V^\bot =0$ and $T\subset \h_2^{\iu}$. By maximality ($\h_2^{\iu}\cap\m_4^V$ is a Lie triple system), 
$T=\h_2^{\iu}\cap\m_4^V$ with $\iu\in V$, and we have found a subalgebra as desired.

Second,  in the situation b), take $\ell$ such that $V^\bot\cap W=\R\ell$. Arguing as above, $d(\ell)\in V\cap W=0$ for all $d\in T=\h_4^{W}\cap\m_4^V$ so that $T\subset \h_2^{\ell}$. Again we have finished since the maximality forces
$T=\h_2^{\ell}\cap\m_4^V$ with $\ell\in V^\bot$. We get here a contradiction, since  $\dim(\h_4^{W}\cap\m_4^V)=4$  but   $\dim(\h_2^{\ell}\cap\m_4^V)=5$, so that both LTS cannot be equal.

Third,  in the situation c), take $\ell$ such that  $V^\bot\cap W^\bot=\R \ell$.
Thus, $d(\ell)\in V\cap W^\bot=0$ for all $d\in T$ and we proceed as in the above case ($\ell\in V^\bot$).
    \end{proof}

  \begin{remark}
  {\rm
  There seems to be a missing case: clearly $\m_4^W \cap  \m_4^V$ is a Lie triple subsystem of $\m_4^V$, and it is maximal since it is conjugated to 
  $\h_4^W \cap  \m_4^V$.  Conjugated but different. Why does it not appear in the list given by   Theorem~\ref{pr_LTS_MSv1}?
  The answer is that it does appear, but it is camouflaged. 
    If $V$ and $W$ are associative subalgebras as in the above proof, 
  and $\iu$, $\ju$ and $\ell$ are as in Equations~\eqref{eq_parademo1} and \eqref{eq_parademo2},
  take $F\in G_2$ determined by $F(\iu)=\iu$, $F(\ju)=\ju$ and $F(\ell)=\ju\t\ell$, as in Section~\ref{sub_M_S}. (There exists such an automorphism of $(\R^7,\t)$ 
  because both $\{\iu,\ju,\ell\}$ and $\{\iu,\ju,\ju\t\ell\}$ are orthonormal systems such that the third vector is orthogonal to the associative subalgebra
  generated by the  other two.)  
  Now $F(V^\bot\cap W)=V^\bot\cap W^\bot$, $F(V^\bot\cap W^\bot)=V^\bot\cap W$, $F(V\cap W^\bot)=V\cap W^\bot$ and $F(V\cap W)=V\cap W$. Taking into account Eq.~\eqref{eq_trocitoauxiliarv2}, we get $\m_4^W \cap  \m_4^V=\h_4^{F(W)} \cap  \m_4^V$, where $F(W)\in\M_S$ intersects both $V$ and $V^\bot$.
 } \end{remark}
 
 %%%%%%%%%%%%%%%%%%%%%%%%%%%%%%%%%%%%%%%%%%%%
  
\subsection{Maximal totally geodesic submanifolds of $\M_S$}

Now it is very easy to find the submanifolds of the symmetric space with tangent spaces those in Theorem~\ref{pr_LTS_MSv1}.

\begin{theorem}\label{pr_tgm_MSv1}  
Any maximal  totally geodesic submanifold of $\M_S$ is one of the following:

\noindent of dimension 2,
 \begin{itemize} 
  \item[(1)] $\N_1^\h:=\{ V\in\M_S\colon \h\textrm{ adapted to }  V \}$, for some principal TDS $\h$; 
  \end{itemize}
  of dimension 5,
  \begin{itemize} 
  \item[(2)] $\N_2^u:=\{V\in\M_S\colon u\in V^\bot\}$, for a fixed $0\ne u\in \R^7$;
   \end{itemize}
  and, of dimension 4,
  \begin{itemize} 
  \item[(3)] 
             $\N_3^u:=\{V\in\M_S\colon u\in V \}$, for a fixed $0\ne u\in \R^7$;  \vspace{2pt}
        \item[(4)]   $\N_4^W :=\{ V\in\M_S\colon W\cap V\ne0, W^\bot\cap V\ne0\}$, for a fixed $W\in\M_S$.   

   \end{itemize}
 \end{theorem}

 \begin{proof}
  To compute the tangent spaces, it can be useful to recall that $G_2$ is compact, so that the exponential is surjective and $G_2=\exp(\g_2)$. If $V$ belongs to a submanifold $\N$ of $\M_S$,
 then $T_V\N=\{d\in T_V\M_S=\m_4^V:\exp(td)(V)\in\N\ \forall t\in\R\}$. 
 
\boxed{\textrm{Case }(1)} 
First observe that Proposition~\ref{pr_lasadaptadas} allows us to write the first manifold as
$$
\N_1^\h=\{F(V_0)\colon F\in H^\h\},
$$
for $V_0\in\M_S$ a fixed associative subalgebra to which $\h$ is adapted. For any $V\in\N_1^\h$,
similarly we   have $\N_1^\h=\{F(V)\colon F\in H^\h\}$, and 
 we want to prove that
$T_V\N_1^\h=\h \cap  \m_4^V$ (two-dimensional, just because $\h$ is adapted to $V$).
First, for any $d\in\h$, and any $t\in\R$, we have $[td,\h]\subset\h$ since $\h$ is subalgebra. 
Thus $\hbox{Ad}(\exp(td))(\h)=\exp(\hbox{ad}(td))(\h)\subset \h$, that implies $\exp(td)\in H^\h$ 
so that  $\exp(td)(V)\in \N_1^\h$, as required.
To prove the converse, take $d$ an odd derivation with $\exp(td)(V)\in \N_1^\h$. Then there is $F_t\in H^\h$ such that
$\exp(td)(V)=F_t(V)$, with $F_0=\id_{\R^7}$. This gives a curve  $F_t^{-1}\exp(t d)$ in $H^V$ ($\cong\SO(V^\bot,\esc{\cdot,\cdot })$) since the automorphism fixes $V$ and it is an isometry. Deriving at $t=0$ we get that $\frac{d}{dt}F_t^{-1}\vert_{t=0} +d\in\h_4^V$, since $\h_4^V$ is the Lie algebra of $H^V$. Its projection on the odd part of the $\Z_2$-grading $\g_2=\h_4^V\oplus\m_4^V$ is zero. 
That is, for $p_V$ the projection as in Lemma~\ref{le_utiladaptada}, $d=-p_V\big(\frac{d}{dt}F_t^{-1}\vert_{t=0} \big)\in p_V(\h)\subset \h$ since $\h$ is adapted to $V$ (and $F_t\in H^\h$).

 \boxed{\textrm{Case }(2)} 
  Take $V\in\M_S$ and $\ell\in V^\bot$.
 Let us prove that $T_V\N_2^\ell=\h_2^\ell\cap  \m_4^V=\{d\in\g_2:d(V)\subset V^\bot, d(V^\bot)\subset V, d(\ell)=0\}$.
 This is equivalent to show that, given $d$ an odd derivation,   $d(\ell)=0$ if and only if $\exp(td)(V)\in\N_2^\ell$ for all $t$, that is, if and only if
 $\ell\in \exp(td)(V)^\bot$ for all $t$. One of the implications is clear: if $d(\ell)=0$, then $\exp(td)(\ell)=\sum_{n=0}^\infty\frac{t^n}{n!}d^n(\ell)=\ell$ belongs to $
 \exp(td)(V^\bot)= \exp(td)(V)^\bot$. Conversely, if $\ell\in \exp(td)(V ^\bot)$ for all $t$, then $d(\ell)\in V^\bot$, but $\ell\in V^\bot$ and $d$ is odd, so that $d(\ell)\in V$ too, what gives $d(\ell)=0$.

 \boxed{\textrm{Case }(3)} 
 This works similarly to the above case. Now  $V\in\M_S$ and ${\iu}\in V$. We want to check that 
 $T_V\N_3^{\iu}=\h_2^{\iu}\cap  \m_4^V=\{d\in\g_2:d(V)\subset V^\bot, d(V^\bot)\subset V, d({\iu})=0\}$.
 (Observe that the appearance is the same, but it is not the same set,   not even of the same dimension.)
 Now we have to prove that  $d({\iu})=0$, for $d$ an odd derivation, if and only if ${\iu}\in \exp(td)(V) $ for all $t$.
 Again, if $d({\iu})=0$, then ${\iu}=\exp(td)({\iu})$, now belonging to $\exp(td)(V) $.
 Conversely, on one hand, if ${\iu}\in \exp(td)(V  )$ for all $t$, then $d({\iu})\in V $, but on the other hand, ${\iu}\in V$ and $d$ is odd, so that
 $d({\iu})\in V^\bot$ and $d({\iu})=0$ follows.
 
 \boxed{\textrm{Case }(4)}  
 For this case, note that, if $V,W\in\M_S$, then 
 \begin{center}$W\cap V\ne0$ and $ W^\bot\cap V\ne0$ if and only if 
 $V=(W\cap V)\oplus (W^\bot\cap V) $.   \end{center}
 In fact, consider $\iu$, $\ell$ and $\ju$ as in Eqs.~\eqref{eq_parademo1} and \eqref{eq_parademo2}.
 Moreover, $V\in \N_4^W$ if and only if $W\in\N_4^V$.

 Let us prove that, for $d$ an odd derivation (odd with respect to a fixed $V\in \N_4^W$), then
 $\exp(td)(V)\in \N_4^W$ for all $t$ if and only if $d\in\h_4^W \cap  \m_4^V$, that is, if $d(W)\subset W$  ($d$ is even with respect to $W$).
 One of the implications is clear: if $d(W)\subset W$, of course  $\exp(td)(W)\subset W$, and $\exp(td)(W^\bot)\subset W^\bot$.
 As $0\ne W\cap V$, and taking into account that $\exp(td)$ is an automorphism, 
 also $0\ne \exp(td)(W\cap V)=\exp(td)(W)\cap \exp(td)(V)=W\cap \exp(td)(V)$. Similarly,
 $0\ne \exp(td)(W^\bot\cap V)=\exp(td)(W^\bot)\cap \exp(td)(V)=W^\bot\cap \exp(td)(V)$, so that $\exp(td)(V)\in \N_4^W$.

In order to prove the converse, take, for each $V\in\M_S$, the order two automorphism 
$
\theta_V\in G_2
$
in Eq.~\eqref{eq_order2auto}, given by $\theta_V\vert_V=\id$ and $\theta_V\vert_{V^\perp}=-\id.$
Observe that we can write
\begin{equation}\label{pasapa}
\N_4^W=\{V\in\M_S:\theta_W(V)\subset V\}\setminus\{W\}=\{V\in\M_S :\theta_W\theta_V=\theta_V\theta_W\}\setminus\{W\}.
\end{equation}
Indeed, for any $V\in\N_4^W$, take $\iu$, $\ell$ and $\ju$ as in Eqs.~\eqref{eq_parademo1} and \eqref{eq_parademo2} and we know  that $\theta_W$ leaves $V$ invariant ($\theta_W(\iu)=\iu$, $\theta_W(\ju)=-\ju$ and $\theta_W(\iu\t\ju)=-\iu\t\ju$). The converse is clear from the fact that the projections on $W$ and on $W^\bot$ are $\frac12(I\pm\theta_W)$. The second equality holds since two commuting automorphisms diagonalize simultaneously.

Now assume $\exp(td)(V)\in \N_4^W$ for all $t$.
A straightforward computation gives $\theta_{\exp(td)(V)}=\hbox{Ad}(\exp(td))(\theta_V)$, so that
$$\exp(td) \theta_V\exp(-td) \theta_W= \theta_W\exp(td) \theta_V\exp(-td).
$$
 Making $\frac{d}{dt}\vert_{t=0}$, we get
$d\theta_V\theta_W-\theta_Vd\theta_W=\theta_Wd\theta_V-\theta_W\theta_Vd$. As $d$ is an odd derivation,  
$d$ anticommutes with $\theta_V$, but taking into account that $\theta_V$ commutes with $\theta_W$,   we can write $2d\theta_W\theta_V=2\theta_Wd\theta_V$. Being $\theta_V$ bijective, this implies that $d$ commutes with $\theta_W$ and hence that $d(W)\subset W$.
\end{proof}

\begin{remark}\label{re_loscocientesquesalen}
{\rm
Observe that the submanifolds above can be seen quite naturally as homogeneous spaces 
   $$
   \begin{array}{cc}
\quad\N_1^\h\cong\SO(3)/\SO(2),  \quad
&\quad\N_2^{u}\cong\SU(3)/\SO(3),  \quad \\
\quad\N_3^{u}\cong\SU(3)/\textrm{U}(2),  \quad
&\quad\N_4^W\cong \SO(4)/\SO(2)\times\SO(2).\quad
\end{array}
$$ 
In particular, $\N_1^\h\cong\mathbb S^2$,  $\N_2^{u}$ is a Wu manifold \cite{Wumf},  $\N_3^{u}\cong\mathbb C P^2$ and 
$\N_4^W\cong(\mathbb S^2\times \mathbb S^2)/\mathbb Z_2$.  
Indeed, taking in mind Proposition~\ref{pr_tgm_casoG2}, 
\begin{itemize}
\item
$H^\h\cong\SO(3)$ acts (transitively) on $\N_1^\h$ and the isotropy subgroup of $V\in\N_1^\h$ is $H^V\cap H^\h
=\{F\in G_2:F(V^\bot)\subset V^\bot,\hbox{\textrm{Ad}}(F)(\h)=\h\} \cong\SO(\h\cap\m_4^V,\kappa)\cong\SO(2)$, where the isomorphism is given by
$F\mapsto \hbox{Ad}(F)\vert_{\h\cap\m_4^V}$.
\item
For $\ell\in V^\bot$, $H^\ell\cong\SU(3)$ acts (transitively) on $\N_2^\ell$ and the isotropy subgroup of $V\in\N_2^\ell$ is $H^V\cap H^\ell=\{F\in G_2:F(V^\bot)\subset V^\bot,F(\ell)=\ell\}$, which is naturally isomorphic to $\SO(V')\cong\SO(3)$ for the three-dimensional space $V'=\{v\t\ell:v\in V\}\le V^\bot$ by the restriction map $F\mapsto F\vert_{V'}$.
\item Again $H^{\iu}\cong\SU(3)$ acts (transitively) on $\N_3^{\iu}$ and the isotropy subgroup of $V\in\N_3^{\iu}$ is $H^V\cap H^{\iu}=\{F\in G_2:F(V^\bot)\subset V^\bot,F({\iu})={\iu}\}$. But now ${\iu}\in V$, so that $ H^V\cap H^{\iu}\cong\textrm{U}(V^\bot,\sigma\vert_{V^\bot})\cong\textrm{U}(2)$ for $\sigma$ the Hermitian form 
$\sigma(x,y)=\esc{x,y}-\esc{{\iu}x,y}{\iu}\in\C_{\iu}$
as in the proof of Proposition~\ref{pr_LTS_G2}. Note $V^\bot\cong (\C_{\iu})^2.$
\item
Finally, $H^W\cong\SO(4)$ acts (transitively) on $\N_4^W$ and the isotropy subgroup of $V\in\N_4^W$ is $H^V\cap H^W=\{F\in G_2:F\theta_V=\theta_V F,F\theta_W=\theta_W F\}\cong \SO(V\cap W^\bot)\t\SO(V^\bot\cap W)\cong \SO(2)\t\SO(2)$, with the isomorphism given by the two restrictions. (Note that $V\cap W^\bot$, $V^\bot\cap W^\bot$ and $V^\bot\cap W $ are the three non-trivial homogeneous components of the $\Z_2^2$-grading on $\R^7$ produced by the pair of commuting order 2 automorphisms $\theta_V$ and $\theta_W$, and the product of  any two of them is the third one. 
Also the square of any of the three components fills the one-dimensional subspace $V\cap W$.
In particular the pair $(F\vert_{V\cap W^\bot},F\vert_{V^\bot\cap W})$, and in general any of the three pairs of possible restrictions, determines $F$.)
\end{itemize}
}
\end{remark}

   \begin{remark}    
   {\rm
  So far, there was not a big difference between using $\R^7$ or using $\O$. But, with the alternative description of $\M_S$ as the quaternion subalgebras of $\O$; it turns out that the submanifold of the fourth type does not consist of those quaternion subalgebras   intersecting both $Q$ and $Q^\bot$, for $Q$ a fixed quaternion subalgebra. 
We would have had to ask for some requirement on the dimension of the intersections, which would make this description less geometrical.  
}
\end{remark}

%%%%%%%%%%%%%%%%%%%%%%%%%%%%%%%%%%%%%%%%%%%%%%%%%%%%%%%%

\section{Maximal totally geodesic submanifolds of $G_2/\SO(4)$ (octonion-free version)}  \label{se_stg_MS_free}

For the sake of completeness, we would like to include a version of our results independent from $\g_2$, or at least where the 
involved Lie triple systems are described in matricial terms. Our interest lies in making the description of the Lie triple systems 
 as accessible as possible,  
 without the need to use derivations (derivations involve quadratic equations). Thus, 
 this section can be read independently of the previous ones,  
 although in this context it would be more difficult to independently prove that we are dealing with all maximal LTS.
 
 %%%%%%%%%%%%%%%%%%%%%%%%%%%%%%%%%%%%%%%%%%%%%%%%%%%%%%%%

\subsection{Lie triple system related to $\M'_S$}

Recall we have defined the Grassmanian as the manifold
\begin{equation*}  
\grass\nolimits_3(\R^7)=\{\pi\in \hbox{End}_{\R}(\R^7)\colon \pi^2=\pi, \pi^\sharp=\pi, \tr(\pi)=3\}.
\end{equation*}
This makes possible to realize the computations of the tangent spaces by using dual numbers.
Define the Grassmanian as a functor mapping each real commutative associative and unital algebra $A$ to the set $\grass_3(A^7)$ of all $\pi\in\hbox{End}_A(A^7)$ such
that $\pi^2=\pi$, $\pi^\sharp=\pi$ and $\tr(\pi)=3$. So taking $A=\R$ we recover the
original Grassmanian. But now we can take $A$ to be any real algebra with the mentioned specifications (commutative associative and unital). So  we can put $A=\R(\epsilon)$, the algebra of dual numbers over $\R$. Recall that $A=\R 1\oplus\R\epsilon$ with $\epsilon^2=0$.
Any element in $\grass_3(\R(\epsilon)^7)$ is of the form $a+\epsilon b$ where 
$a,b\in\hbox{End}_{\R}(\R^7)$.
Then if we take $\pi\in\grass_3(\R^7)$ and  $d\in\hbox{End}_{\R}(\R^7)$ we may 
try to see under what conditions $\pi+\epsilon d$ is in $\grass_3(\R(\epsilon)^7)$.
In other words we can define 
$$T_\pi(\grass\nolimits_3(\R^7)):=\{d\in\hbox{End}_{\R}(\R^7)\colon \pi+\epsilon d\in\grass\nolimits_3(\R(\epsilon)^7)\}.$$
This means 
$(\pi+\epsilon d)^2=\pi+\epsilon d$; $(\pi+\epsilon d)^\sharp=\pi+\epsilon d$ and $\tr(\pi+\epsilon d)=3.$
The first assertion is $\pi+\epsilon d=(\pi+\epsilon d)^2=\pi^2+\epsilon(d\pi+\pi d)$
and since $\pi^2=\pi$ we get $d\pi+\pi d=d$. The second equality is 
$\pi+\epsilon d=\pi^\sharp+\epsilon d^\sharp$ and since $\pi^\sharp=\pi$ we get $d=d^\sharp$. Finally $3=\tr(\pi+\epsilon d)=\tr(\pi)+\epsilon\tr(d)$ from which we deduce that $\tr(d)=0$. Then we have checked that 
$$
T_\pi(\grass\nolimits_3(\R^7))=\{d\in\hbox{End}_{\R}(\R^7)\colon d\pi+\pi d=d, d^\sharp=d, \tr(d)=0\}.
$$
Notice that each $\pi\in\grass_3(\R^7)$  induces a $\Z_2$-grading on the vector space $\R^7=V_0\oplus V_1$  being $V_{\lambda}=\ker(\pi'-\lambda\id)$, with $\dim V_0=3$. 
  As usual $\pi'$ denotes $1-\pi$.
By (vector space) grading we   simply mean a decomposition of $\R^7$ as a direct sum of vector subspaces (ortogonal sum, in this case). 
Now the condition $d\pi+\pi d=d$ is equivalent to the assertion that $d$ is an odd map relative to  such grading:
$d(V_0)\subset V_1$ and $d(V_1)\subset V_0$.  
It is easy to check that $\dim T_\pi(\grass_3(\R^7))=12$ because for any linear map $f\colon V_0\to V_1$, there is an only selfadjoint extension $d\colon \R^7\to\R^7$ with $d\vert_{V_0}=f$ (which is trivially traceless).
In terms of matrices, fix any basis of $\R^7$ with the first 3 vectors in $V_0$ and the remaining 4 vectors in $V_1$, so that the   $7\times 7$ matrix   of $\pi$ is $\tiny\begin{pmatrix}I_3 & 0_{34}\\0_{43} & 0_{44}\end{pmatrix}$ where $I_3$ is the identity matrix $3\times 3$ and $0_{ij}$ is a null $i\times j$ block. Then identifying $d$ with a block matrix and imposing the conditions $d\pi+\pi d=d$ and $d^\sharp =d$ (the condition $\tr(d)=0$ is redundant), we see that the matrix of $d$ becomes  $\tiny\begin{pmatrix}0 & b\\b^t &0\end{pmatrix}$ for $b$  any $3\times 4$ block. 

Now $T_\pi(\grass\nolimits_3(\R^7))$ should be a Lie triple system, since $\grass\nolimits_n(\R^7)$ is always a symmetric space. 
Let us check that  the Lie triple system of the tangent space to the Grassmanian $\grass_3(\R^7)$ is the skew-symmetrization of the associative triple system of $\mat_{3,4}(\R)$ (basic concepts on Section~\ref{se_basicosonLTS}). 
\begin{lemma} The tangent space $T_\pi(\grass_3(\R^7))$ is a Lie triple system for the product $[d_1,d_2,d_3]=[[d_1,d_2],d_3]$.
Moreover, $T_\pi(\grass\nolimits_3(\R^7))\cong\mat\nolimits_{3,4}(\R)^-.$
\end{lemma}
\begin{proof} The condition $\tr([d_1,d_2,d_3])=0$ is trivially satisfied and since 
$[d_1,d_2]^\sharp=[d_2^\sharp,d_1^\sharp]=[d_2,d_1]$, then taking into account the anti-commutativity of the Lie bracket, we have $[d_1,d_2,d_3]^\sharp=[d_1,d_2,d_3]$. Now, for any $d_i\in T_{\pi}(\grass_3(\R^7))$ we have 
$$d_1d_2d_3\pi=d_1d_2(d_3-\pi d_3)=d_1d_2d_3-d_1d_2\pi d_3=$$
$$=d_1d_2d_3-d_1(d_2-\pi d_2)d_3=d_1\pi d_2d_3=d_1d_2d_3-\pi d_1d_2d_3,$$
whence $d_1d_2d_3\pi+\pi d_1d_2d_3=d_1d_2d_3$ and from this, it is easy to realize that 
$[d_1,d_2,d_3]\pi+\pi[d_1,d_2,d_3]=[d_1,d_2,d_3]$. This could also have been proved taking into account the equivalent condition that $d_i(V_0)\subset V_1$ and $d_i(V_1)\subset V_0$ for $i=1,2,3$.

If we use the identification between $T_\pi(\grass_3(\R^7))$  with $ \mat_{3,4}(\R)$ above proposed,  the induced  triple product  follows from
\begin{equation}\label{precise}
[[\begin{pmatrix}0 & a\\a^t &0\end{pmatrix},\begin{pmatrix}0 & b\\b^t &0\end{pmatrix}],\begin{pmatrix}0 & c\\c^t &0\end{pmatrix}]=\begin{pmatrix}0 & ab^tc-ba^tc+cb^ta-ca^tb\\ * &0\end{pmatrix},\end{equation}
where the asterisk denotes the transposition of the upper right element of the matrix. 
That is, the vector space isomorphism $\tiny\begin{pmatrix}0 & a\\a^t &0\end{pmatrix}\mapsto a$ is really a Lie triple system  isomorphism.
\end{proof}

Now think of the submanifold   
$\M'_S=\{\pi\in \grass\nolimits_3(\R^7)\colon \Omega(\pi\o\pi\o\pi')=0\}$. 
If $\pi\in\M'_S$, the $\Z_2$-grading on the vector space $\R^7$ induced by $\pi$, that is, $V_{\bar i}=V_i=V_{i,\pi'}$, is in fact a grading on the   algebra $(\R^7,\t)$ as in  \eqref{eq_Z2grad},
since $V_{\bar i}\t V_{\bar j}\subset V_{\overline{i+j}}$ for all $\bar i,\bar j\in\Z_2$. 
To obtain the tangent space at
$\pi\in\M'_S$, we use again dual numbers. If $d\in  T_\pi(\grass\nolimits_3(\R^7))$, the condition to belong to $  T_{\pi}(\M'_S)$ is
$\Omega ((\pi+\epsilon d)\otimes(\pi+\epsilon d)\otimes(\pi'-\epsilon d))=0$, 
which gives
$\Omega (d\o\pi\o\pi'+\pi\o d\o\pi'-\pi\o\pi\o d)=0 $.
As $d^\sharp =d$, also $(\pi')^\sharp =\pi'$, and $\esc{\cdot,\cdot }$ is nondegenerate, we get
\begin{equation}\label{idnoc}
\pi'\left[d(x)\t \pi(y)+\pi(x)\t d(y)\right]=d(\pi(x)\times \pi(y)).
\end{equation}
This condition  is trivially satisfied  if $x\in V_\0$, $y\in V_\1$, since can be written as 
$ \pi'\left[x\t d(y)\right]=0$, but $x\times d(y)\in V_\0\times V_\0\subset V_\0$. The same can be said in the cases $x\in V_\1, y\in V_\0$ or $x,y\in V_\1$. So we can assure that 
$$
T_\pi(\M'_S)=\{d\in T_\pi(\grass\nolimits_3(\R^7)): d(x\times y)=d(x)\times y+x\times d(y)\ \forall x,y\in V_\0=\hbox{Im}(\pi)\}.
$$

\begin{remark}\label{re_noestamal}
{\rm
Observe that the tangent space to the symmetric space $T_\pi(\M'_S)$ does not coincide with the odd elements in $\g_2=\der(\R^7,\t)$:
The elements in $T_\pi(\grass\nolimits_3(\R^7))$ are self-adjoint, while the elements in $\g_2$ are skew-adjoint.
This may lead to confusion. But observe that both $T_\pi(\M'_S)$
and the set of odd derivations of $(\R^7,\t)$ are isomorphic as Lie triple systems! Given $d\in T_\pi(\M'_S)$, there is a unique $\tilde d\in\m_4^{V_\0}\le\g_2$ such that $\tilde d\vert_{V_\0}=d$ (and of course this natural map $d\mapsto\tilde d$ behaves well with the triple product).
It is even more striking if we think of $\M_S'$ as $\{\pi\in \grass\nolimits_4(\R^7)\colon \Omega(\pi\o\pi\o\pi)=0\}$. Then it is easy to check that the tangent space is $\{d\in T_\pi(\grass\nolimits_4(\R^7)): -d(x\times y)=d(x)\times y+x\times d(y)\ \forall x,y\in  \hbox{Im}(\pi)\}$, 
whose elements do not even have the appearance of derivations.
}\end{remark}

Since $d^\sharp=d$, the restriction
$d\colon V_\0\to V_\1$ contains all the relevant information. 
Fix orthonormal and unitary $\iu,\ju\in V_\0$ and put $\ku=\iu\times \ju$. Then $\{\iu,\ju,\ku\}$ is an orthonormal basis of $V_\0$. Choose now $\ell\in V_\0^\bot$ of norm $1$ and define $v_0:=\ell$, $v_1:=\iu\times\ell$, $v_2:=\ju\times\ell$, $v_3:=\ku\times\ell$. 
\begin{lemma}
    For any two vectors $(a_i), (b_i)\in\R^4$, there is a 
    unique $d\in T_\pi(\M'_S)$ such that $d(\iu)=\sum_{i=0}^3a_iv_i$ and $d(\ju)=\sum_{i=0}^3b_iv_i$. So $\dim(T_\pi(\M'_S))=8$. Furthermore
    $T_\pi(\M'_S)$ is a Lie triple subsystem of $T_\pi(\grass_3(\R^7))$.
\end{lemma}
\begin{proof}
    Define the linear map $d\colon V_\0\to V_\1$ such that 
    $d(\iu)=\sum_ia_iv_i$, $d(\ju)=\sum_ib_iv_i$ and $d(\ku)=\iu\times d(\ju)+d(\iu)\times\ju$. 
     To get the explicit   coordinates of $d(\ku)$, take into account the following rules from \eqref{props_cross}:
    $$
    \begin{tabular}{|c|c|c|c|c|}
    \hline 
    $\times$ & $v_0$ & $v_1$ & $v_2$ & $v_3$\\
    \hline 
    $\iu$ & $v_1$ & $-v_0$ & $-v_3$ & $v_2$\\
    $\ju$ & $v_2$ & $v_3$ & $-v_0$ & $-v_1$\\
    $\ku$ & $v_3$ & $-v_2$ & $v_1$ & $-v_0$\\
    \hline
    \end{tabular}
    $$
     Thus we get $$d(\ku)=(a_2-b_1)v_0+
    (a_3+b_0)v_1-(a_0-b_3)v_2-(a_1+b_2)v_3,$$
    which allows us to directly check that  $d(\iu\times \ku)=\iu\times d(\ku)+d(\iu)\times\ku$ and 
    $d(\ju\times \ku)=\ju\times d(\ku)+d(\ju)\times\ku$, so that $d\in T_\pi(\M'_S)$.  In particular, the (row) matrix of a generic element of $T_\pi(\M'_S)$ relative to the bases
    $\{\iu,\ju,\ku\}$ of $V_\0$ and $\{v_0,v_1,v_2,v_3\}$ of $V_\1$ is 
    \begin{equation}\label{ortaucropsert}
        {\tiny\begin{pmatrix} a_0 & a_1 & a_2 & a_3\\
b_0 & b_1 & b_2 & b_3\\
a_2-b_1 & a_3+b_0 & b_3-a_0 & -a_1-b_2\end{pmatrix}.}
    \end{equation}
    
    To check that $T_\pi(\M'_S)$ is a subtriple  of   $T_\pi(\grass_3(\R^7))\cong\mat\nolimits_{3,4}(\R)^-$, recall that the Lie triple product  is $[a,b,c]=ab^tc-ba^tc+cb^ta-ca^tb$, as given in \eqref{precise}. The reader can get convinced of that the triple product is closed by testing directly the fact  on a basis, for instance on $\{e_{11}-e_{33},e_{12}-e_{34},e_{13}+e_{31},e_{14}+e_{32},e_{21}+e_{32},e_{22}-e_{31},e_{23}-e_{34},e_{24}+e_{33}\}$.
    Here $e_{ij}$ denotes the $3\t4$ block matrix with the only nonzero entry, the $(i,j)$th, equal to 1.
    \end{proof}

If now we take $d\colon V_\0\to V_\1$ as before satisfying
$d(x\times y)=x\times d(y)+d(x)\times y$, then the endomorphism 
$\hat d\colon V_\0\to V_\0$ given by $\hat d(x)=\ell\times d(x)$ has (relative to $\{\iu,\ju,\ku\}$) the matrix
\begin{equation}\label{osopardo}
{\tiny\begin{pmatrix}a_1 & a_2 & a_3\\ b_1 & b_2 & b_3\\ a_3+b_0 & b_3-a_0 & -a_1-b_2\end{pmatrix}},
\end{equation}
recalling that $\ell\t(v\t\ell)=v$ for all $v\in V_\0$.
This is a generic element of $\sl_3(\R)$ so there is an isomorphism of Lie triple systems $T_\pi(\M'_S)\cong\sl(V_\0)$ such that $d\mapsto\hat d$, that in matricial terms is no more than \lq\lq forgetting the first column\rq\rq:
\begin{equation}\label{truco}
{\tiny\begin{pmatrix} a_0 & a_1 & a_2 & a_3\\
b_0 & b_1 & b_2 & b_3\\
a_2-b_1 & a_3+b_0 & b_3-a_0 & -a_1-b_2\end{pmatrix}\mapsto \begin{pmatrix}a_1 & a_2 & a_3\\ b_1 & b_2 & b_3\\ a_3+b_0 & b_3-a_0 & -a_1-b_2\end{pmatrix}}.
\end{equation}
We have to consider in $\sl_3(\R)$ the unique triple product 
making of $d\mapsto\hat d$ a Lie triple system isomorphism.
The linear maps $L^+\colon V_\0\to V_\1$ and $L^-\colon V_\1\to V_\0$ such that $L^+(u_0)=\ell\times u_0$ and
$L^-(u_1)=\ell\times u_1$ satisfy $L^-L^+=-1\colon V_\0\to V_\0$. The inverse of the above  bijective linear map is given by $   \sl(V_\0)\to T_\pi(\M'_S)$ such that $ f\mapsto -L^+f$.
So the triple product $\{\cdot,\cdot,\cdot\}$ in $\sl_3(\R)$ given by
$L^+\{f_1,f_2,f_3\}=[L^+f_1,L^+f_2,L^+f_3]$ 
 endows $\sl_3(\R)$ with a LTS structure isomorphic to the one on $T_\pi(\M'_S)$. 
 Consequently 
$$\{f_1,f_2,f_3\}=-L^-[L^+f_1,L^+f_2,L^+f_3].$$
  If we construct the map $\a\colon\sl_3(\R)\to (\R^3)^t$ such that 
  $
  \tiny\a\left[(a_{ij})\right]:=\begin{pmatrix}a_{23}-a_{32}\\ a_{31}-a_{13}\\a_{12}-a_{21}\end{pmatrix},
$ then
  \begin{comment}
      $$
  \tiny\a\left[\begin{pmatrix}a_{11} &a_{12} & a_{13}\\ a_{21} &a_{22} & a_{23}\\
a_{31} &a_{32} & a_{33}\end{pmatrix}\right]:=\begin{pmatrix}a_{23}-a_{32}\\ a_{31}-a_{13}\\a_{12}-a_{21}\end{pmatrix},
$$
\end{comment} 
then $L^+ m=-(\a(m)\vert m)$ and $L^-(\a(m)\vert m)=m$, for all $m\in\sl_3(\R)$, using block matrices.   Thus 
$$
\begin{array}{l} 
-(L^+m_1)(L^+m_2^t)(L^+m_3)=(\a(m_1)\vert m_1)\begin{pmatrix}\a(m_2)^t\\m_2^t\end{pmatrix}(\a(m_3)\vert m_3)\\
\qquad =(\a(m_1)\a(m_2)^t\a(m_3)+m_1m_2^t \a(m_3)\vert \a(m_1)\a(m_2)^t m_3+m_1m_2^tm_3);
\end{array}
$$
and
 $-L^-  ((L^+m_1)(L^+m_2^t)(L^+m_3)) =\a(m_1)\a(m_2)^t m_3+m_1m_2^tm_3$. Summing and interchanging roles among $m_i$'s,
 we get 
\begin{equation}\label{excplts}
\{m_1,m_2,m_3\}=[m_1,m_2,m_3]+\gamma(m_1,m_2,m_3)\end{equation}
where 
\begin{align}\label{doslts}
    \begin{split}
    [m_1,m_2,m_3]= & m_1m_2^tm_3-m_2m_1^tm_3+m_3m_2^tm_1-m_3m_1^tm_2,\\
    \gamma(m_1,m_2,m_3)= &[\a(m_1)\a(m_2)^t-\a(m_2)\a(m_1)^t]m_3+\a(m_3)\a(m_2)^tm_1-\\
    & \a(m_3)\a(m_1)^tm_2.
    \end{split}
\end{align}

Notice that this reveals two of the LTS structures on $\sl_3(\R)$. One is the given by the triple product $[ \cdot ,\cdot ,\cdot ]$ above \eqref{doslts},   and another one, $\{\cdot ,\cdot ,\cdot \}$ \eqref{excplts}, corresponds to the odd part of $\g_2$ with its $\Z_2$-grading. 
Summarizing we have
\begin{proposition}\label{pr_LTSconenvolventeg2}
The tangent space $T_\pi(\M'_S)$ is isomorphic to $\sl_3(\R)$ provided with its LTS structure \eqref{excplts}.   
Thus, 
$(\sl_3(\R),\{\cdot ,\cdot ,\cdot \} )$  is the exceptional LTS whose standard envelope is $\g_2$.
\end{proposition}

 Now we describe the maximal Lie triple subsystems  in these matricial terms, jointly with their related maximal totally geodesic submanifolds.
 
%%%%%%%%%%%%%%%%%%%%%%%%%%%%%%%%%%%%%%%%%%%%%%%%%%%%%%%%

\subsection{The $2$-dimensional LTS}\label{se_2dimLTS_v2}

First, we adapt the concepts in Section~\ref{se_adapted} to the setting of projections and Grassmannians. 
Thus, if $\h$ is a principal subalgebra of $\g_2$,    and $\pi\in \M'_S$, we say that $\h$ is \emph{adapted to $\pi$}
if $\h$ is adapted to $V_\0=\hbox{Fix}(\pi)$, that is, if  $\h$ is graded relative to the $\Z_2$-grading produced by $\pi$:
any $d\in\h$ can be written as $d=d_\0+d_\1$, with $d_i\in\h $  
 and $ d_i(V_j) \subset V_{i+j}$ for any $i,j=\bar0,\bar1$. 
\begin{comment}
 \color{blue} If $\theta:=\pi-\pi'=2\pi-1$ then $\theta$ is an order-two automorphism and 
 $d=\frac{1}{2}(d-\theta d\theta)+\frac{1}{2}(d+\theta d\theta)$ where $d_0=\frac{1}{2}(d-\theta d\theta)$ and $d_1=\frac{1}{2}(d+\theta d\theta)$. Then $2d_1=d+(2\pi-1)d(2\pi-1)\in\h$
so that $\h\ni d+4\pi d\pi-2\pi d-2 d\pi +d$ and finally $2\pi d\pi-\pi d-d \pi\in\h$. Equivalently $\h$ is adapted to $\pi$ if and only if $\h$ is invariant under $L_\pi+R_\pi-2L_\pi R_\pi=(L_\pi-R_\pi)^2   $.
\color{black} 
\end{comment}
Recall that, fixed a principal $\h$, we can always find $\pi_0\in\M'_S$ adapted to $\h$. 
This allows us to define, for $\pi_0$ adapted to $\h$,    
$$
{\N'_1}^\h:=\{F\pi_0 F^{-1}:F\in H^\h\}.
$$  
This set does not depend on the chosen $\pi_0$, because Theorem~\ref{pr_tgm_MSv1} and its proof imply that we can write the
manifold without reference to the principal group, as
${\N'_1}^\h= \{\pi\in\M'_S\colon \pi\textrm{ is adapted to }\h\}.$
According to the previous section,
${\N'_1}^\h$   is a totally geodesic submanifold of $\M'_S$ of dimension 2, because
for 
any $\pi\in {\N'_1}^\h$, we have
$T_\pi{\N'_1}^\h=\h\cap T_\pi\M'_S$ (2-dimensional since $\pi$ is adapted to $\h$), maximal Lie triple subsystem with envelope equal to $\h$. 
 Let us describe the LTS in matricial terms as a subtriple of $\sl_3(\R)$.
 For instance, take $\h=\esc{\{h_1,h_2,h_3\}}$ in Eq.~\eqref{eq_unappal}, and compute
$$\begin{array}{c}
\frac1{\sqrt6}h_2: 
\iu\mapsto 2\iu\t\ell, 
\  \ju\mapsto -\ju\t\ell-\sqrt{{\scriptstyle\frac53}}\ell,\  
\ku\mapsto (-\ku-\sqrt{{\scriptstyle\frac53}}\iu)\t\ell,\\
\frac1{\sqrt6}h_3:\iu\mapsto -2 \ell, \  \ju\mapsto (-\ku+\sqrt{{\scriptstyle\frac53}}\iu)\t\ell,\   \ku\mapsto \ju\t\ell-\sqrt{{\scriptstyle\frac53}}\ell.
\end{array}
$$
Thus, the row-matrix descriptions of the elements in $\h\cap T_\pi\M'_s=\esc{h_2,h_3}$ are
\begin{equation}\label{eq_ppalenterminosdematrices}
T_\pi({\N'_1}^\h)=\{
d_{s,t}:={\tiny\begin{pmatrix}  -2 s & 0 & 0 \\
  -\sqrt{\frac{5}{3}} t & s & t \\
   \sqrt{\frac{5}{3}} s & -t & s \\\end{pmatrix}}
\colon s,t\in \R\}\le \sl_3(\R).
\end{equation}
One could have proved directly that this is a Lie triple subsystem, simply by checking that
\begin{equation}\label{eq_triplenesfera} 
\{d_{s_{1},t_{1}},d_{s_{2},t_{2}},d_{s_{3},t_{3}}\}=\frac23(s_1t_2-s_2t_1)d_{t_{3},-s_{3}}.
\end{equation}
(On the other hand, this is an immediate  consequence of Eq.~\eqref{eq_formulitabasica}.)

Finally we provide a description of ${\N'_1}^\h$   that refers as little as possible to $\g_2$. 
(Take into account that $\h$ is adapted to $V$ if $\h$  is a graded subspace relative to the $\Z_2$-grading on $\g_2$ induced by $\theta_V$.)
To be more explicit, the purpose is simply to find a self-contained way of expressing that $\h$  is adapted to $\pi\in\M'_S$,
writing the projections on the components of the $\Z_2$-grading induced by $\pi\in\M'_S$ on the Lie algebra $\mathfrak{gl}(7,\R)$  in terms on the own $\pi$. 
(As usual, the even component $\mathfrak{gl}(7,\R)_{\bar 0}$ consists of the linear maps $d$ such that $d(V_i)\subset V_i$ and the odd component 
$\mathfrak{gl}(7,\R)_{\bar 1}$ consists of the linear maps $d$ such that $d(V_i)\subset V_{i+1}$.)
 Clearly, the projection on the odd part is given by
$$
p_\pi\colon \mathfrak{gl}(7,\R)\to \mathfrak{gl}(7,\R)_{\bar 1},\quad p_\pi(d)=\pi d \pi'+\pi' d \pi.
$$
For any vector subspace $\mathfrak{s}$ of $\mathfrak{gl}(7,\R)$, $\mathfrak{s}$ will be a graded subspace, that is,
$\mathfrak{s}=(\mathfrak{s}\cap \mathfrak{gl}(7,\R)_{\bar 0})\oplus (\mathfrak{s}\cap \mathfrak{gl}(7,\R)_{\bar 1})$, if and only if $p_\pi(\mathfrak{s})\subset \mathfrak{s}$. So the set of  projections adapted to $\h$ can be described as
$$
{\N'_1}^\h=\{\pi\in\M'_S: p_\pi(\h)\subset\h\}.
$$

\begin{remark}
According to Proposition~\ref{pr_LTSconenvolventeg2}, this 2-dimensional manifold is a sphere. 
One can check that it has constant curvature, as expected, simply passing the scalar product in $\mat\nolimits_{3,4}(\R)^-$ to $\sl_3(\R)$
through the isomorphism in \eqref{truco}.     
This gives in $\sl_3(\mathbb R)$ the scalar product (up to constant) 
\begin{equation}\label{eq_metricase3}
 \esc{d,d'}=\alpha(d)^t\alpha(d')+\mathrm{tr}(d(d')^t)
\end{equation}
  for $d,d'\in\sl_3(\mathbb R)$. 
Now $R(d_{s_{1},t_{1}},d_{s_{2},t_{2}},d_{s_{3},t_{3}})= -\frac23(s_1t_2-s_2t_1)d_{t_{3},-s_{3}}$  by \eqref{eq_triplenesfera} and Section~\ref{se_curvatura},
but 
$$
\esc{d_{s_{1 },t_{ 1}},d_{s_{ 3},t_{ 3}}}d_{s_{2 },t_{ 2}}-\esc{d_{s_{2 },t_{ 2}},d_{s_{3 },t_{3 }}}d_{s_{1 },t_{ 1}}=
-\frac{28}{3} (s_1t_2-s_2t_1) d_{t_3,-s_3}.
$$ 
\end{remark}

 %%%%%%%%%%%%%%%%%%%%%%%%%%%%%%%%%%%%%%%%%%%%%%%%%%%%

\subsection{The $5$-dimensional LTS}
For each  unit element $u \in\R^7$,
consider $${\N'_2}^u:=\{\pi\in\M'_S\colon \pi(u)=0\}.$$
For a fixed $\pi\in{\N'_2}^u$, the tangent space $T_\pi({\N'_2}^u)$ is just the space of all 
$d\in T_\pi(\M'_S)$ such that $\pi+\epsilon d\in {\N'_2}^u(\R(\epsilon))$. 
Thinking of the elements of the tangent space as linear maps we have: 
$0=(\pi+\epsilon d)(u)=\pi(u)+\epsilon d(u)$. Thus 
$$
T_\pi({\N'_2}^u)=\{d\in T_\pi(\M'_S):d(u)=0   \}.
$$
To write this LTS in terms of our matricial representations,  note that $d=d^\sharp$ gives $0=n(d(u),x)=n(u,d(x))$ for any $x\in V_\0$. 
Consequently $\hbox{Im}(d)\in u^\bot$ and so the first column in \eqref{ortaucropsert} must be zero because 
$ u\in V_\1$ can be taken as $v_0=\ell$ without loss of generality. Then $a_0=b_0=a_2-b_1=0$
and under the isomorphism $T_\pi(\M'_S)\cong\sl_3(\R)$ given in \eqref{truco}, the  triple subsystem $T_\pi({\N'_2}^\ell)$ transforms isomorphically to 
\begin{equation}\label{eq_triple2}
\begin{array}{ll}
T_\pi({\N'_2}^\ell)&=\{
{\tiny\begin{pmatrix}a_1 & a_2 & a_3\\a_2 & b_2 & b_3\\a_3 & b_3 &-a_1-b_2\end{pmatrix}}
\colon a_i,b_j\in \R\} \vspace{2pt}\\
&=\{A\in \mat\nolimits_{3,3}(\R):A=A^t,\textrm{tr}(A)=0\},
\end{array}
\end{equation}
that is, the subtriple of traceless symmetric matrices.
It is not difficult to check directly that this LTS is maximal in $\sl_3(\R)$ relative to the triple product \eqref{excplts} (which in this case agrees with the product $[\cdot,\cdot,\cdot]$ in \eqref{doslts} because of the symmetry of $T_\pi({\N'_2}^\ell)$).

\subsection{The two $4$-dimensional LTS}
For each  unit element $u\in\R^7$, consider
\begin{equation*}
    {\N'_3}^u := \{\pi\in\M'_S\colon \pi(u)=u\}.
\end{equation*}
Again for any $\pi\in {\N'_3}^u$ we have 
$$
T_\pi( {\N'_3}^u)=\{d\in T_\pi(\M'_S)\colon d(u)=0\}.
$$
This is not the same Lie triple system as above:   now $u\in V_{0,\pi'}$, so that $u$ plays the role of $\iu$, 
while in the above case, $u\in V_{1,\pi'}$  played the role of $\ell$. 

So, returning to \eqref{truco},  
if $d(\iu)=0$ we get $a_0=a_1=a_2=a_3=0$, so that
\begin{equation}\label{eq_triple3}
T_\pi( {\N'_3}^{\iu})\cong{\tiny
\left\{
\begin{pmatrix} 0 & 0 & 0\\ b_1 & b_2 & b_3\\ b_0 & b_3 & -b_2\end{pmatrix}\colon b_0, b_1, b_2, b_3\in\R\right\}.}
\end{equation} \smallskip

%........................................

We still have to find the reflective LTS. For any 
\begin{comment}
a division quaternion algebra $\H$ with canonical basis $\{1,\iu,\ju,\ku\}$. Having fixed an octonion algebra $\O=\H\oplus \H\ell$ we may consider $\O':=(\C\oplus\C\ell)\oplus(\C\oplus\C\ell)\ju$ where $\C=\R 1\oplus\R\iu$. Since $\C\oplus\C\ell$ is a quaternion algebra one can consider its associated projection $\sigma\colon\R^7\to\R^7$ (as usual $1^\bot=\R^7$) such that $Fix(\sigma)=1^\bot\cap(\C\oplus\C\ell)=\R\iu\oplus\R\ell\oplus\R(\iu\times\ell)$.
Note that $\ker(\s)=\R\ju\oplus\R\ku\oplus\R(\ju\times\ell)\oplus\R(\ku\times\ell)$.
\end{comment}
  $\sigma\in \M'_S$, consider the centralizer of $\sigma$ in $\M'_S$:
\begin{equation*}
    {\N'_4}^\sigma:=\{\pi\in\M'_S\colon \pi\sigma=\s\pi\}
\end{equation*}
Arguing again with dual numbers, for a fixed $\pi\in{\N'_4}^\sigma$,  
 the tangent space  at $\pi$ is just the space of all 
$d\in T_\pi(\M'_S)$ such that $\pi+\epsilon d $ commutes with $\sigma$, that is, $d$ does.  
Hence,
$$
T_\pi({\N'_4}^\sigma)=\{d\in T_\pi(\M'_S)\colon d\s=\s d\}.
$$  

In order to write it in terms of matrices, take $  \iu$ a unit vector in $V_{0,\pi'}\cap V_{0,\sigma'}$,  $  \ju$ a unit vector in $V_{0,\pi'}\cap V_{1,\sigma'}$,
and $  \ell$ a unit vector in $V_{1,\pi'}\cap V_{0,\sigma'}$. 
If we denote, as usual, $\ku:=\iu\t\ju$, then the eigenspaces for $\s$,  $\R\iu\oplus\R\ell\oplus\R(\iu\times\ell)$ and $\R\ju\oplus\R\ku\oplus\R(\ju\times\ell)\oplus\R(\ku\times\ell)$, are $d$-invariant for any $d$ commuting with $\s$.
 Thus, for any $d\in T_\pi({\N'_4}^\sigma)$, 
 we have
 $d(\iu)\in V_{1,\pi'}\cap V_{0,\sigma'}=\langle\{\ell, \iu\t\ell\}\rangle$ so $a_2=a_3=0$;
 and $d(\ju)\in V_{1,\pi'}\cap V_{1,\sigma'}=\langle\{\ju\t\ell,\ku\t\ell\}\rangle$, so that $b_0=b_1=0$.
That is, the matrix of $d$ in \eqref{ortaucropsert} is of the form 
 \begin{equation*}\label{ortauc1}   
       { \tiny\begin{pmatrix} a_0 & a_1 & 0 & 0\\
0 & 0 & b_2 & b_3\\
0 & 0 & b_3-a_0 & -a_1-b_2\end{pmatrix}}.
    \end{equation*}
 Under the isomorphism $T_\pi(\M'_S)\cong\sl_3(\R)$ given in \eqref{truco}, 
the triple subsystem $T_\pi({\N'_4}^\sigma)$ transforms isomorphically to 
\begin{equation}\label{atalpzep} 
T_\pi({\N'_4}^\sigma)\cong\left\{{\tiny\begin{pmatrix}  s_1 & 0 & 0\\
 0 & s_2 & s_3\\
 0 & s_4 & -s_1-s_2\end{pmatrix}}\colon s_i\in\R\,\forall i=1\dots 4\right\}.
 \end{equation}
Note that this LTS is reflective since its orthogonal (relative to the scalar product in $\sl_3(\mathbb R)$ given by  
 \eqref{eq_metricase3}) is    
 \begin{equation}\label{gotro}
 \left\{{\tiny\begin{pmatrix}  0 & s_1 & s_2\\
 s_3 & 0 & 0\\
 s_4 & 0 & 0\end{pmatrix}}\colon s_i\in\R\,\forall i=1\dots 4\right\},
 \end{equation}
 which is easily seen to be another Lie triple subsystem of $\sl_3(\R)$. (Moreover,  the LTS in \eqref{gotro} and \eqref{atalpzep} are conjugated.)

 %%%%%%%%%%%%%%%%%%%%%%%%%%%%%%%%%%%%%%%%%%%%%%%%

\subsection{Conclusions}

We can give a second version of our main results, Theorem~\ref{pr_tgm_MSv1} and Theorem~\ref{pr_LTS_MSv1}, providing the maximal totally geodesic submanifolds  of  
$\M'_S=\{\pi\in \grass\nolimits_3(\R^7)\colon \Omega(\pi\o\pi\o\pi')=0\}$, that is, describing all these submanifolds in terms of  Grassmannians in a natural way, as well as the related maximal Lie triple subsystems of the LTS of traceless $3\t3$ matrices.

  \begin{corollary}\label{co_v2LTS}
    The maximal Lie triple  systems of  the exceptional LTS $(\sl_3(\R),\{\cdot,\cdot,\cdot \} )$ are, up to conjugation by an element in $\SO(3)$, the given ones  in 
    \eqref{eq_ppalenterminosdematrices},   \eqref{eq_triple2}, \eqref{eq_triple3}   and \eqref{gotro}. 
    \end{corollary}
    Simply recall that we have given the matricial descriptions in terms of a fixed  orthonormal basis of $V_0$, but we could have chosen any other orthonormal basis.
    
\begin{corollary}\label{co_v2}
    If $\N$ is a maximal totally geodesic submanifold of $\M'_S$, then $\N$ equals either

   \begin{enumerate}
     \item  $ {\N'_1}^\h=\{\pi\in\M'_S: p_\pi(\h)\subset\h\}$, for some principal subalgebra $\h $; or
        \item   $ {\N'_2}^u=\{\pi\in\M'_S\colon \pi(u)=u\}$ or $ {\N'_3}^u
        =\{\pi\in\M'_S\colon \pi(u)=0\}$, for some   unit vector $u\in\R^7$; or
        \item   $ {\N'_4}^\sigma=\{\pi\in\M'_S\colon \pi\sigma=\s\pi\}$, for some $\s\in\M'_S$.
        \end{enumerate}
\end{corollary}

\subsection*{Acknowledgement}

The authors would like to thank Alfonso Romero, Alberto Rodr\'\i guez V\'azquez and the various referees for their useful suggestions and encouraging reading.
%%%%%%%%%%%%%%%%%%%%%%%%%%%%%%%%

 \bibliographystyle{plain}

\bibliography{ref}
\end{document}